\documentclass[12pt]{article}
\usepackage{hyperref,amsmath,amsthm,amsfonts,amssymb,amscd,enumerate,tikz}

\usetikzlibrary{positioning}

\theoremstyle{plain}
\newtheorem{theorem}{Theorem}[section]
\newtheorem{proposition}[theorem]{Proposition}

\newtheorem{corollary}[theorem]{Corollary}
\newtheorem{lemma}[theorem]{Lemma}

\newtheorem{Lemma}[theorem]{Lemma}

\newtheorem*{Question}{Question}
\DeclareMathOperator*{\Aster}{\text{\LARGE{\textasteriskcentered}}}
\DeclareMathOperator{\aster}{\text{\LARGE{\textasteriskcentered}}}

\DeclareMathOperator{\piprod}{\raisebox{-0.1em}{\huge{$\pi$}}\kern -0.2em}

\newtheorem*{I(n)}{$\I(n)$}
\newtheorem*{III(2k+1)}{$\III(2k+1)$}
\newtheorem*{III'(2k+1)}{$\III'(2k+1)$}
\newtheorem*{V(n)}{$\V(n)$}

\theoremstyle{definition}

\newtheorem{example}[theorem]{Example}

\newtheorem{remark}[theorem]{Remark}

\newtheorem*{notation}{Notation}

\newtheorem*{example*}{Example}
\newtheorem*{Remark}{Remark}
\newcommand{\ca}{\mathcal {A}}
\newcommand{\cb}{\mathcal {B}}
\newcommand{\cac}{\mathcal {C}}

\newcommand{\ch}{\mathcal {H}}
\newcommand{\ci}{\mathcal {I}}
\newcommand{\ck}{\mathcal {K}}
\newcommand{\cl}{\mathcal {L}}
\newcommand{\cn}{\mathcal {N}_{\bq}}
\newcommand{\cnp}{\mathcal {N}_{\bp}}
\newcommand{\cN}{\mathcal {N}}
\newcommand{\cp}{\mathcal {P}}
\newcommand{\car}{\mathcal {R}}
\newcommand{\cs}{\mathcal {S}}

\newcommand{\cu}{\mathcal {U}}
\newcommand{\cv}{\mathcal {V}}

\newcommand{\FF}{{\mathbb F}}

\newcommand{\bd}{{\mathbf D}_\infty}
\newcommand{\zz}{{\mathbf Z}}

\newcommand{\ltwo}{L^2_{\mathbf {q}}}
\newcommand{\ltwop}{L^2_{\mathbf {p}}}

\newcommand{\Ltwo}{L^2_q}
\newcommand{\mup}{\mu_{\mathbf {p}}}
\newcommand{\muq}{\mu_{\mathbf {q}}}

\newcommand{\minus}{{-1}}

\newcommand{\eltwo}{{\ell^2}}
\newcommand{\rH}{{\ol {H}}}

\newcommand{\bR}{{\mathbf R}}
\newcommand{\bZ}{{\mathbf Z}}

\newcommand{\bC}{{\mathbf C}}
\newcommand{\bD}{{\mathbf D}}

\newcommand{\bp}{{\mathbf p}}
\newcommand{\bq}{{\mathbf q}}

\newcommand{\bs}{{\mathbf s}}
\newcommand{\bt}{{\mathbf t}}
\newcommand{\bone}{{\mathbf 1}}

\newcommand{\h}{\hat}
\newcommand{\wt}{\widetilde}
\newcommand{\wx}{\widetilde {X}}

\newcommand{\ol}{\overline}

\newcommand{\rh}{{\ol{H}\,}}
\def\clap#1{\hbox to 0pt{\hss#1\hss}}

\newcommand{\norm}[1]{\lVert#1\rVert}

\newcommand{\comment}[1]{}

\newcommand{\ga}{\alpha}

\newcommand{\gd}{\delta}

\newcommand{\gf}{\varphi}

\newcommand{\gs}{\sigma}

\newcommand{\gt}{\tau}

\newcommand{\gG}{\Gamma}

\newcommand{\OL}{{OL}}

\newcommand{\Aut}{\operatorname{Aut}}
\newcommand{\vertex}{\operatorname{Vert}}
\newcommand{\edge}{\operatorname{Edge}}

\newcommand{\flag}{\operatorname{Flag}}
\newcommand{\Hom}{\operatorname{Hom}}

\newcommand{\Ker}{\operatorname{Ker}}
\newcommand{\Ima}{\operatorname{Im}}
\newcommand{\Lk}{\operatorname{Lk}}

\newcommand{\cat}{\operatorname{CAT}}
\newcommand{\cone}{\operatorname{Cone}}

\newcommand{\tr}{\operatorname{tr}}

\newcommand{\pim}{\pi^\minus (J)}

\newcommand{\GR}{\operatorname{Gr}}
\newcommand{\grh}{\GR{H}}

\newcommand\mapright[1]{\smash{\mathop{\longrightarrow}\limits^{#1}}}

\newenvironment{enumerate1}{
\begin{enumerate}[\upshape (1)]}%
	{
\end{enumerate}
}
\newenvironment{enumerate1*}{
\begin{enumerate}[\upshape (*1)]}%
	{
\end{enumerate}
}
\newenvironment{enumeratei}{
\begin{enumerate}[\upshape (i)]}%
	{
\end{enumerate}
}

\newenvironment{enumeratea'}{
\begin{enumerate}[\upshape (a)$'$]}{
\end{enumerate}
}

\DeclareMathOperator{\I}{\bf I}
\DeclareMathOperator{\III}{\bf III}
\DeclareMathOperator{\V}{\bf V}

 \numberwithin{equation}{section} 

\begin{document}
\title{Cohomology computations for Artin groups, Bestvina-Brady groups, and graph products}
\author{Michael W. Davis\thanks{The first author was partially supported by NSF grant DMS 0706259.} \and{Boris Okun} }
\date{\today} \maketitle
\begin{abstract}
	We compute:
	\begin{itemize}
		\item the cohomology with group ring coefficients of Artin groups (or actually, of their associated Salvetti complexes), Bestvina--Brady groups, and graph products of groups,
		\item the $L^2$-Betti numbers of Bestvina--Brady groups and of graph products of groups,
		\item the weighted $L^2$-Betti numbers of graph products of Coxeter groups.
	\end{itemize}
	In the case of arbitrary graph products there is an additional proviso: either all factors are infinite or all are finite.
	(However, for graph products of Coxeter groups this proviso is unnecessary.) \smallskip

	\paragraph{AMS classification numbers.} Primary: 20F36, 20F55, 20F65, 20J06, 55N25 \\
	Secondary: 20E42, 57M07. \smallskip

	\paragraph{Keywords:} Artin group, Bestvina--Brady group, building, Coxeter group, graph product, right-angled Artin group, $L^2$-Betti number, weighted $L^2$-cohomology.
\end{abstract}

\section*{Introduction}\addcontentsline{toc}{section}{Introduction}

This paper concerns the calculation of the group cohomology, $H^*(G;N)$, for certain discrete groups $G$, where the $G$-module $N$ is either $\zz G$ or a von Neumann algebra $\cN(G)$.
Here $\cN (G)$ is a completion of the group algebra $\bR G$ acting on $L^2(G)$, the Hilbert space of square summable functions on $G$.
If $G$ acts properly and cocompactly on a CW complex $Y$, then one has the reduced $L^2$-cohomology spaces, $L^2H^*(Y)$.
These are Hilbert spaces with orthogonal $G$-actions.
As such, each has a ``von Neumann dimension'' with respect to $\cN(G)$.
The dimension of $L^2H^n(Y)$ is the $n^{th}$ $L^2$-\emph{Betti number}, $L^2b^n(Y;G)$.
When $Y$ is acyclic, it is an invariant of $G$, denoted $L^2b^n(G)$.
As shown in \cite{luck}, these $L^2$-Betti numbers can be computed from the cohomology groups $H^*(G;\cN(G))$ (i.e., the cohomology of $BG$ with local coefficients in $\cN(G)$).

For Coxeter groups, there is a refinement of the notion of $L^2$-Betti number.
Suppose $(W,S)$ is a Coxeter system and $\bp$ is a multiparameter of positive real numbers (meaning that $\bp$ is a function $S\to (0,\infty)$ which is constant on conjugacy classes).
There is an associated ``Hecke--von Neumann algebra,'' $\cN_\bp(W)$, which one can use to define the ``weighted $L^2$-Betti numbers,'' $\ltwop b^n(W)$, cf.~\cite{dym}, \cite{ddjo07} or \cite{dbook}.
These numbers have been calculated provided $\bp$ lies within a certain range, namely, when either $\bp$ or $\bp^\minus$ lie within the region of convergence for the growth series of $W$ (in \cite{dym}, \cite{ddjo07}).

Associated to $(W,S)$, there is a finite CW complex $X$, called its \emph{Salvetti complex}.
It is homotopy equivalent to the quotient by $W$ of the complement of the (possibly infinite) complex hyperplane arrangement associated to $(W,S)$ (cf.~\cite{cd05}).
The fundamental group of $X$ is the \emph{Artin group} $A$ associated to $(W,S)$.
The $K(\pi,1)$ Conjecture for Artin groups asserts that $X$ is a model for $BA$ ($=K(A,1)$).
This conjecture is known to hold in many cases, for example when $W$ is either right-angled or finite.

Given a simplicial graph $\gG$ with vertex set $S$ and a family of groups $\{G_s\}_{s\in S}$, their \emph{graph product}, $\prod_\gG G_s$, is the quotient of the free product of the $G_s$ by the relations that elements of $G_s$ and $G_t$ commute whenever $\{s,t\}$ is an edge of $\gG$.
Associated to $\gG$ there is a flag complex $L$ with $1$-skeleton $\gG$.
($L$ is defined by the requirement that a subset of $S$ spans a simplex of $L$ if and only if it is the vertex set of a complete subgraph of $\gG$.) A \emph{right-angled Coxeter group} $W_L$ (a RACG for short) is a graph product where each $G_s$ is cyclic of order $2$.
Similarly, a \emph{right-angled Artin group} (a RAAG for short) $A_L$ is a graph product where each $G_s$ is infinite cyclic.
An arbitrary graph product of groups is an example of a right-angled building of type $(W_L, S)$.
In \S\ref{ss:oct} we consider a family of arbitrary Coxeter systems $\{(V_s,T_s)\}_{s\in S}$.
Their graph product, $V:=\prod_\gG V_s$, gives a Coxeter system $(V,T)$ (where $T$ denotes the disjoint union of the $T_s$).

Given a right-angled Artin group $A_L$, let $\pi:A_L\to \zz$ be the homomorphism which sends each standard generator to $1$.
Put $BB_L:=\Ker \pi$.
In \cite {bb} Bestvina and Brady prove that $BB_L$ is type FP if and only if the simplicial complex $L$ is acyclic.
If this is the case, $BB_L$ is called a \emph{Bestvina--Brady group}.

A subset of $S$ (the set of generators of a Coxeter group $W$) is \emph{spherical} if it generates a finite subgroup of $W$.
Let $\cs(W,S)$ denote the poset of spherical subsets of $S$ and let $K$ be the geometric realization of this poset.
For each spherical subset $J$, let $K_J$ (resp.
$\partial K_J$) be the geometric realization of the subposet $\cs(W,S)_{\ge J}$ (resp.
$\cs(W,S)_{> J}$).
$K_J$ is the cone on $\partial K_J$.
($\partial K$ is the barycentric subdivision of the nerve of $(W,S)$; $\partial K_J$ is the barycentric subdivision of the link of the simplex corresponding to $J$ in the nerve.)

Many of the following computations are done by using a spectral sequence associated to a double complex.
The $E_\infty$ terms of such a spectral sequence compute the graded group, $\grh^*(\ )$, associated to the corresponding filtration of the cohomology group, $H^*(\ )$, in question.

In item (1) below, the Coxeter system is arbitrary while in (2), (3), (4) it is right-angled.
(Within parentheses we refer either to the theorem in this paper where the computation appears or else we give a reference to the literature.) Here are the computations.
\begin{enumerate1}
	\item Suppose $A$ is the Artin group associated to a Coxeter system $(W,S)$, $X$ is the associated Salvetti complex and $\wt{X}$ is its universal cover.
	Then
	\begin{enumerate}
		\item (Theorem~\ref{t:salvetti}).
		\[
			\grh^n(X;\zz A)=\bigoplus _{J\in \cs(W,S)} H^{n-|J|}(K_J,\partial K_J) \otimes H^{|J|}(A_J;\zz A).
		\]
		(In the case of a RAAG this formula is the main result of \cite{jm}.)
		\item (\cite[Cor.~2]{dl}).
		\[
			L^2b^n(\wt{X};A)=b^n(K,\partial K).
		\]
		(Here $b^n(Y,Z)$ is the ordinary Betti number of a pair $(Y,Z)$, i.e., $b^n(Y,Z):=\dim H^n(Y,Z; \bR)$.)
	\end{enumerate}
	When the $K(\pi,1)$ Conjecture holds for $A$, these formulas compute, \\
	$\grh^*(A;\zz A)$ and $L^2b^*(A)$, respectively.

	\item Suppose $BB_L$ is the Bestvina--Brady subgroup of a RAAG, $A_L$.
	\begin{enumerate}
		\item (Theorem~\ref{t:bb}).
		The cohomology of $BB_L$ with group ring coefficients is isomorphic to that of $A_L$ shifted up in degree by $1$:
		\[
			\grh^n(BB_L;\zz BB_L)=\bigoplus_{J\in \cs(W,S)} H^{n-|J|+1}(K_J,\partial K_J) \otimes \zz (A_L/A_J).
		\]
		\item (Theorem~\ref{t:lraag}).
		The $L^2$-Betti numbers of $BB_L$ are given by
		\[
			L^2b^n(BB_L)=\sum_{s\in S} b^n(K_s,\partial K_s),
		\]
	\end{enumerate}

	\item Suppose $G=\prod_\gG G_s$ is a graph product of groups, and let $(W,S)$ be the RACS associated to the graph.
	For each spherical subset $J$, $G_J$ denotes the direct product $\prod_{s\in J} G_s$.
	\begin{enumerate}
		\item (Theorem~\ref{t:gpfree}).
		If each $G_s$ is infinite, then
		\[
			\grh^n(G;\zz G)=\bigoplus_{\substack{J\in \cs(W,S)\\
			i+j=n}} H^i(K_J,\partial K_J; H^j(G_J;\zz G)).
		\]
		\item (Theorem~\ref{t:lgraph}).
		If each $G_s$ is infinite, then
		\[
			L^2b^n(G)=\bigoplus_{\substack{J\in \cs(W,S)\\
			i+j=n}} b^i(K_J,\partial K_J)\cdot L^2 b^j(G_J).
		\]
		\item (\cite[Corollary 9.4]{ddjmo}).
		If each $G_s$ is finite, then
		\[
			\grh^n(G;\zz G)=\bigoplus_{J\in \cs(W,S)} H^n(K,K^{S-J})\otimes \hat{A}^J,
		\]
		Here $K^{S-J}$ denotes the union of subcomplexes $K_s$, with $s\in S-J$, and $\hat{A}^J$ is a certain (free abelian) subgroup of $\zz (G/G_J)$.
		\item (\cite[Theorem 13.8]{ddjo07} and Corollary~\ref{cor:ddjofinite} below).
		Again, when each $G_s$ is finite, $G$ is a right-angled building of type $(W,S)$ and its $L^2$-Betti numbers are determined by the weighted $L^2$-Betti numbers of $(W,S)$ via the formula, $L^2b^n(G)=\ltwop b^n(W)$, where the multiparameter $\bp=(p_s)_{s\in S}$ is defined by $p_s=|G_s|-1$.
	\end{enumerate}
	\item Suppose that $V$ is a graph product of Coxeter groups $\{V_s\}_{s\in S}$, that $(W,S)$ is the RACS associated to the graph and that $\bq$ is multiparameter for $(V,T)$.
	There are two results concerning the weighted $L^2$-Betti numbers of $(V,T)$.
	\begin{enumerate}
		\item (Theorem~\ref{t:largeq}).
		Suppose $\bq$ is ``large'' in the sense that it does not lie in the region of convergence for the growth series of any component group $V_s$.
		Then
		\[
			\ltwo b^n(V)=\sum_{\substack{i+j=n\\J\in \cs}} b^i(K_J,\partial K_J)\cdot \ltwo b^j(V_J),
		\]
		\item (Theorem~\ref{c:main}).
		For $V$ as above and $\bq$ ``small'' in the sense that it lies in the convergence region of each $V_s$, then
		\[
			\ltwo b^*(V)=\ltwop b^*(W),
		\]
		where the multiparameter $\bp$ for $W$ is defined by $p_s = V_s(\bq) -1$.
	\end{enumerate}
\end{enumerate1}

The calculations in (1), (2), (3a), (3b) and (4a) follow a similar line.
They are based on the spectral sequence developed in \S\ref{s:ss}.
In all cases we are computing some type of cohomology of a CW complex $Y$, which is covered by a family of subcomplexes $\{Y_J\}_{J\in \cs(W,S)}$, indexed by $\cs(W,S)$.
For a fixed $j$, the $E^{*,j}_1$ terms of the spectral sequence form a cochain complex for a nonconstant coefficient system on $K$, which associates to a simplex $\gs$ of $K$ with minimum vertex $J$ the group $H^j(Y_J)$.
We first show that the spectral sequence decomposes at $E_2$ as a direct sum, with a component for each $J\in \cs(W,S)$, and with the $J$-component consisting of a cochain complex of the form $C^*(K_J,\partial K_J)$ with some constant system of coefficients.
We then show the spectral sequence collapses at $E_2$.

Since $\zz G$ is a $G$-bimodule, both sides of the formulas in (1a), (2a) and (3a) are right $G$-modules.
One can ask if these formulas give isomorphisms of $G$-modules.
The spectral sequence argument shows that this is indeed the case.
Moreover, since the right hand side of each formula is finitely generated as a $G$-module, so is the left hand side (cf.~\cite{ddjo2}).  In general, if we change the $\grh^n(-)$ on the left hand side to  $H^n(-)$, then these formulas do not give isomorphisms of $G$-modules.  For example, (3a) is not valid after this change in the case where $G$ is the free product of two infinite cyclic groups (cf.\cite[Example 5.2]{ddjo2}).  However, the possibility remains that after dropping the $\GR$'s, both sides of these equations are still isomorphic as $\zz$-modules, which leads us to the following question.

\begin{Question}
On the left hand sides of the the formulas in (1a), (2a) and (3a), is $\grh^n(-)$ always isomorphic to $H^*(-)$ as an abelian group?
\end{Question}

The calculations in (3c), (3d) and (4b) come from a different direction based on \cite{ddjmo} and \cite{ddjo07}.
In particular, the proof of (4b) uses an argument, similar to one in \cite{ddjo07}, relating the ordinary $L^2$-cohomology of a building to the weighted $L^2$-cohomology of its Coxeter system.

\section{Basic definitions}
\subsection{Coxeter groups, Artin groups, buildings}\label{ss:cox}
Throughout this paper, we will be given as data a simplicial graph $\gG$ with finite vertex set $S$ and edge set, $\edge(\gG)$, together with a labeling of the edges by integers $\ge 2$.
The label corresponding to $\{s,t\}\in \edge(\gG)$ is denoted by $m(s,t)$.
These data give a presentation of a \emph{Coxeter group} $W$ with generating set $S$ and relations:
\[
	s^2=1,\ (st)^{m(s,t)}=1, \ \ \text{for all } s\in S \text{ and }\{s,t\}\in \edge(\gG).
\]
The pair $(W,S)$ is a \emph{Coxeter system}; $\gG$ is its \emph{presentation graph}.
These same data determine a presentation for an \emph{Artin group} $A$ with generating set $\{a_s\}_{s\in S}$ and relations,
\[
	a_s a_t \cdots = a_t a_s \cdots, \qquad \text{for all $\{s,t\} \in \edge(\gG)$,}
\]
where there are $m(s,t)$ terms on each side of the equation.
Given a word $\bs=(s_1,\dots,s_n)$ in $S$, its \emph{value} is the element $w(\bs)$ of $W$ defined by
\[
	w(\bs):=s_1\cdots s_n.
\]
The word $\bs=(s_1,\dots,s_n)$ is a \emph{reduced expression} if $\bs$ is a word of minimum length for its value.
The pair $(W_J,J)$ is also a Coxeter system (cf.~\cite[Thm.~4.1.6]{dbook}).
For any $J\subseteq S$, the \emph{special subgroup} $W_J$ is the subgroup generated by $J$.
The subset $J$ is \emph{spherical} if $|W_J|<\infty$.
Let $\cs(W,S)$ denote the poset of spherical subsets of $S$.
The nonempty elements of $\cs(W,S)$ form an abstract simplicial complex $L(W,S)$, called the \emph{nerve} of the Coxeter system.
($\vertex(L(W,S))=S$ and a nonempty subset $J\subseteq S$ spans a simplex if and only if it is spherical.)

Given a subset $I\subseteq S$, an element $w\in W$ is \emph{$I$-reduced} if $l(sw)>l(w)$ for all $s\in I$.
For $I\subseteq J\subseteq S$, let $W^I_J$ be the subset of all $I$-reduced elements in the special subgroup $W_J$.
(For example, $W^\emptyset _J=W_J$ and $W^J_J=\{1\}$.)

A \emph{chamber system over a set $S$} is a set $\cac$ (of ``chambers'') and a family of equivalence relations $\{\sim_s\}_{s\in S}$ on $\cac$ indexed by $S$.
(There is one equivalence relation for each $s\in S$.) An $s$-equivalence class is an \emph{$s$-panel}.
Distinct chambers $C,D\in \cac$ are \emph{$s$-adjacent} if they belong to the same $s$-panel.
A \emph{gallery} in $\cac$ is a sequence $\bC=(C_0,\dots, C_n)$ of adjacent chambers.
Its \emph{type} is the word $\bs=(s_1,\dots, s_n)$ in $S$ where the $i^{th}$ letter of $\bs$ is $s_i$ if $C_{i-1}$ be $s_i$-adjacent to $C_i$.
A \emph{building of type $(W,S)$} is a chamber system $\cac$ over $S$ equipped with a function $\gd:\cac\times\cac\to W$ (called a \emph{Weyl distance}) such that
\begin{enumerate1}
	\item Each panel contains at least two elements.
	\item Given a reduced expression $\bs$ and chambers $C,D\in \cac$, there is a gallery of type $\bs$ from $C$ to $D$ if and only if $\gd(C,D)=w(\bs)$.
\end{enumerate1}
(The above definition of building is due to Ronan and Tits, a variant can be found in \cite{ab}.) The building $\cac$ is \emph{locally finite} if each panel is finite.
\begin{example*}
	A Coxeter group $W$ can be given the structure of a chamber system by declaring the $s$-panels to be the left cosets, $wW_s$, where $W_s$ ($=W_{\{s\}}$) is the cyclic group of order two generated by $s$.
	Define $\gd_W:W \times W\to W$ by $\gd(w,w')=w^\minus w'$.
	Then $(W,\gd_W)$ is a building, called the \emph{standard thin building of type $(W,S)$}.
\end{example*}

Given a building $(\cac,\gd)$ of type $(W,S)$ and a subset $J$ of $S$, the \emph{$J$-residue} containing a chamber $C$ is the subset $R_J(C)\subset \cac$ defined by
\[
	R_J(C):=\{D\in \cac\mid \gd(C,D)\in W_J\}.
\]
In other words, a $J$-residue is a ``$J$-gallery connected component of $\cac$.'' If $J=\emptyset$, then a $J$-residue is simply a chamber and if $J$ has only one element, then a $J$-residue is a panel.
In the standard thin building $(W,\gd_W)$, a $J$-residue is a left coset of $W_J$.

\subsection{Posets, simplicial complexes, flag complexes}\label{ss:poset}
Suppose $L$ is a simplicial complex with vertex set $S$ and let $\cs(L)$ denote the poset of (vertex sets of) simplices in $L$ (including the empty simplex).
If $J$ is the vertex set of a simplex $\gs$ in $L$, then denote by $\Lk(J,L)$ (or simply $\Lk(J)$ when $L$ is understood), the link of $\gs$ in $L$.
The abstract simplicial complex $\Lk(J)$ has one simplex for each element of $\cs(L)_{>J}$ ($=\{J'\in \cs(L)\mid J'\supset J\}$).

A simplicial complex $L$ is a \emph{flag complex} if any finite, nonempty set of vertices, which are pairwise connected by edges, spans a simplex of $L$.
A simplicial graph $\gG$ determines a flag complex, $L(\gG)$: its simplices are the finite, nonempty sets of vertices which are pairwise connected by edges.

Suppose $\cp$ is a poset.
There is an abstract simplicial complex $\flag(\cp)$ with vertex set $\cp$ and with simplices the totally ordered, finite, nonempty subsets of $\cp$.
We note that $\flag(\cp)$ is a flag complex.
Given a simplex $\gs\in \flag(\cp)$, its least element is its \emph{minimum vertex} and is denoted by $\min (\gs)$.
If $L$ is a simplicial complex, then $\flag (\cs(L)_{>\emptyset})$ can be identified with the barycentric subdivision of $L$.
Similarly, $\flag(\cs(L))$ is the cone on the barycentric subdivision of $L$.
(The empty set provides the cone point.)

\subsection{Davis complexes and Salvetti complexes.}\label{ss:real}
Let $M$ be a topological space.
A \emph{mirror structure} on $M$ \emph{over a set $S$} is a family of subspaces $\{M_s\}_{s\in S}$ indexed by $S$; $M_s$ is the \emph{$s$-mirror} of $M$.
If $M$ is CW complex and each $M_s$ is a subcomplex, then $M$ is \emph{mirrored CW complex}.
For each $x\in M$,
\begin{equation}\label{e:sx}
	S(x):=\{s\in S\mid x\in M_s\}.
\end{equation}
If $M$ is a mirrored CW complex and $c$ is a cell in $M$, then
\[
	S(c):=\{s\in S\mid c\subset M_s\}.
\]
Given a building $\cac$ of type $(W,S)$ and mirrored CW complex $M$ over $S$, define an equivalence relation $\sim$ on $\cac\times M$ by $(C,x)\sim (C', x')$ if and only if $x=x'$ and $C$, $C'$ belong to the same $S(x)$-residue.
Give $\cac$ the discrete topology, $\cac\times M$ the product topology, and denote the quotient space by
\begin{equation}\label{e:real}
	\cb(\cac,M):=(\cac\times M)/\sim
\end{equation}
and call it the \emph{$M$-realization} of $\cac$.
When $\cac$ is the standard thin building $W$, put
\begin{equation}\label{e:real2}
	\cu(W,M):=\cb(W,M)
\end{equation}
and call it the \emph{$M$-realization} of the Coxeter system $(W,S)$.
\begin{Remark}
	In our previous work, e.g., in \cite{dbook}, we used the notation $\cu(\ ,\ )$ to denote a topological realization of either a building or of a Coxeter system.
	However, in what follows we will study Coxeter systems $(V,T)$ which will also have the structure of a RAB over an auxiliary RACS $(W,S)$ and we will want to distinguish the two types of realizations of $(V,T)$ (either as a Coxeter system or as a building).
\end{Remark}

The mirror structure is \emph{$W$-finite} if $S(x)$ is spherical for each $x\in M$.
(In this paper we shall always assume this.) When this is the case, $\cu(W,M)$ is a locally finite complex and similarly, if $\cac$ is locally finite building, then $\cb(\cac,M)$ is a locally finite complex.

We will use the following notation for unions and intersections of mirrors, for any $J\subseteq S$, put
\begin{equation}\label{e:xj}
	M_J:=\bigcap_{s\in J} M_s,\qquad M^J:=\bigcup_{s\in J} M_s.
\end{equation}
Also, we will write $\partial M_J$ for the subset of $M_J$ consisting of all points $x \in M$ such that $S(x)$ is a proper subset of $J$.

As in \S\ref{ss:poset}, $\cs(W,S)$ (or simply $\cs$) is the poset of spherical subsets of $S$ (including the empty set).
The geometric realization of this poset is the simplicial complex $\flag(\cs(W,S))$.
We denote it by $K(W,S)$ (or simply $K$) and call it the \emph{Davis chamber}.
Most often we will want to take $M=K$.
One gets a mirror structure on $K$ by defining $K_s$ to be the geometric realization of the subposet $\cs_{\ge \{s\}}$.
Then $\cu(W,K)$ is the \emph{Davis complex} and $\cb(\cac,K)$ is the \emph{standard realization} of $\cb$.
By construction, the $W$-action on $\cu(W,K)$ is proper (i.e., each isotropy subgroup is finite) and the quotient space $K$ is a finite complex, hence, compact.
Moreover, $\cu(W,K)$ is contractible (by \cite[Thm.~8.2.13]{dbook}).
Note that for any $J\in \cs$, the subcomplex $\partial K_J$ is the barycentric subdivision $\Lk(J)$ (the link in $L$ of the simplex with vertex set $J$).
Also, $K_J$ is the cone on $\partial K_J$ (i.e., $K_J\cong \cone (\Lk(J))$).

\paragraph{The Salvetti complex.} Let $A$ be the Artin group associated to $(W,S)$ and let $q: A\rightarrow W$ denote the natural homomorphism sending $a_s$ to $s$.
There is a set-theoretic section for $q$, denoted by $w\mapsto a_w$, defined as follows: if $s_1\cdots s_n$ is any reduced expression for $w$, then $a_w := a_{s_1}\cdots a_{s_n}$.
As explained in \cite[p.
602]{cdjams}, it follows from Tits' solution to the word problem for Coxeter groups that $w\mapsto a_w$ is well-defined.

Define a partial order on $W\times \cs$ by $(w,I)<(v,J)$ if and only if $I<J$ and $v^\minus w \in W^I_J$ (where $W^I_J$ was defined in \S\ref{ss:cox}).
It is proved in \cite{cd05} that $W\times \cs$ is the poset of cells of a cell complex $X'$ on which $W$ acts freely so that each cell of $X'$ is a Coxeter cell.
($\flag (W\times \cs)$ is the barycentric subdivision of $X'$.) The quotient space $X:=X'/W$ is the \emph{Salvetti complex} of $(W,S)$.
It is known that $\pi_1(X)=A$.
The universal cover of $X$ is denoted by $\wx$.
For each $(w,J)\in W\times \cs$, the flag complex on $(W\times \cs)_{\le (w,J)}$ can be identified with the barycentric subdivision of a Coxeter cell of type $(W_J)$.
(A \emph{Coxeter cell} of type $W_J$ means the convex hull of a generic orbit in the canonical representation of $W_J$; see \cite[\S 7.3]{dbook}.) So, $X'$ (or $X$)) can be given the structure of a CW complex where the cells are Coxeter cells.
In particular, each vertex of $X'$ corresponds to an element of $W\times \cs$ of the form $(w,\emptyset)$ and each $1$-cell of $X'$ corresponds to an element of the form $(w,\{s\})$.
Orient this edge by declaring $(w,\emptyset)$ to be its initial vertex and $(ws,\emptyset)$ its terminal vertex.
Since the $W$-action preserves edge orientations, these orientations pass to the edges of $X$.
Call a vertex $x$ of a cell $C$ of $X$ a \emph{top vertex} of $C$ if each edge of $C$ which contains $x$ points away from $x$ (cf.~\cite[\S 7]{dl}).
One can then explicitly describe CW structure on $X$ as follows.
For each $J\in \cs$, take a Coxeter cell $C_J$ of type $W_J$.
Now for each $I < J$ and each $u\in W^I_J$ glue together $C_I$ and $uC_I$ via the homeomorphism induced by $u$.
The result is denoted $X_J$.
(It is the Salvetti complex for $A_J$ and therefore, a $K(A_J,1)$.) To construct $X$, start with the disjoint union of the $X_J$, for $J\in \cs$, and then use the natural maps to identify $X_I$ with a subcomplex of $X_J$ whenever $I < J$.
This description exhibits $X$ as a ``poset of spaces over $\cs$'' (as defined in \S\ref{s:ss}).
On the level of fundamental groups we know that the inclusion $X_J\to X$ induces the natural injection $A_J\to A$ and that the associated ``simple complex of groups'' is the one discussed in \cite[\S 3]{cdjams}.
Similarly, for each $J\in \cs$, let $\wx_J$ denote the inverse image of $X_J$ in $\wx$.
($\wx_J$ is a disjoint union of copies of the universal cover $X_J$, one copy for each coset of $A_J$ in $A$.) Thus, $\wx$ also has the structure of a poset of spaces over $\cs$.

In the right-angled case, the Salvetti complex has a simple description as a subcomplex of a torus (and we will denote it by $T_L$ instead of $X$).
Let $T^S$ denote the $S$-fold product of copies of $S^1$.
Define
\begin{equation}\label{e:tl}
	T_L:= \bigcup _{J\in \cs(L)} T^J.
\end{equation}
(This is a special case of the ``polyhedral product'' construction discussed in the next subsection and in \cite{bbcg,densuc}.) According to \cite{cd05}, its universal cover $\wt{T}_L$ is a $\cat(0)$-cubical complex and hence, is contractible.

\subsection{Graph products of groups and spaces}\label{ss:spaces}

As in the Introduction, $\gG$ is a simplicial graph and suppose each edge is labeled $2$.
Also, $S=\vertex(\gG)$, $L$ is the flag complex determined by $\gG$ and $(W_L,S)$ is the associated RACS. (We shall generally reserve the notation $W_L$ for the case when the Coxeter group is right-angled with nerve $L$ and similarly, $A_L$ for the RAAG associated to $L$.) Suppose $\{G_s\}_{s\in S}$ is a family of groups indexed by $S$.
The \emph{graph product} of the $G_s$, denoted $\prod_\gG G_s$, is the quotient of the free product of the $G_s$, $s\in S$, by the normal subgroup generated by all commutators of the form, $[g_s,g_t]$, where $\{s,t\}\in \edge(\gG)$, $g_s\in G_s$ and $g_t\in G_t$.
For example, if $\edge(\gG)=\emptyset$, then $\prod_\gG G_s$ is the free product, while if $\gG$ is the complete graph on $S$, then $\prod_\gG G_s$ is the direct sum.

Suppose $G=\prod_\gG G_s$.
We want to see that $G$ has the structure of a RAB of type $(W_L,S)$.
The group $G$ can be given the structure of a chamber system as follows: the $s$-panels are the left cosets $gG_s$, with $g\in G$.
Write $G_s^*$ for $G_s-\{1\}$.
The projections $G_s^* \mapsto s$ induce a map (\emph{not a homomorphism}) $\pi:G \to W_L$, as follows: any element $g \in G$ can be written as $g_{s_1}\cdots g_{s_n}$, with $g_{s_i} \in G_{s_i}^*$ so that $s_1\cdots s_n$ is a reduced expression for an element $w \in W$.
Moreover, $w$ depends only on $g$.
The map $\pi$ sends $g$ to $w$.
The Weyl distance $\gd:G\times G\to W_L$ is defined by $\gd(g,h)=\pi(g^\minus h)$.
The following lemma is easily checked.
\begin{lemma}[{\cite[Ex.~18.1.10]{dbook}}]\label{l:RAB}
	$(G,\gd)$ is a building of type $(W_L,S)$.
\end{lemma}

\paragraph{Polyhedral products.} Suppose, for the moment, that $S$ is the vertex set of an arbitrary simplicial complex $L$.
Let $\{(Z_s, A_s)\}_{s\in S}$ be a family of pairs of spaces indexed by $S$.
For each $J \in \cs(L)$, let $Z'_J$ be the subspace of the product $\prod_{s\in S} Z_s$, consisting of all $S$-tuples $(x_s)_{s\in S}$ such that
\[
	x_s \in
	\begin{cases}
		Z_s, &\text{if $s\in J$},\\
		A_s &\text{if $s\notin J$.}
	\end{cases}
\]
The \emph{polyhedral product} of this family, denoted $\piprod_L (Z_s, A_s)$, is defined to be the following subspace of $\prod_{s\in S} Z_s$:
\begin{equation}\label{e:prod'}
	\textstyle{\piprod_L} \displaystyle{(Z_s, A_s):= \bigcup_{J\in \cs(L)} Z'_J}.
\end{equation}
(This terminology comes from \cite{bbcg}.
In \cite{densuc} it is called a ``generalized moment angle complex.'')
\begin{example}\label{ex:01}
	Suppose each $(Z_s,A_s)=([0,1],0)$.
	Then $\piprod_L ([0,1],0)$ can be identified with $\flag (\cs(L))$ in such a fashion that a standard subdivision of each cube in the polyhedral product is a subcomplex of $\flag (\cs (L))$.
	In particular, if $L$ is the nerve of a RACS, then $\piprod_L ([0,1],0)=K$, the Davis chamber from \S \ref{ss:real}.
\end{example}

\paragraph{Graph products of spaces.} We return to the assumption that $L$ is the flag complex determined by $\gG$.
\begin{example}\label{ex:cone}
	Suppose each $(Z_s, A_s)=(\cone(G_s), G_s)$ for a family of discrete groups $\{G_s\}_{s\in S}$.
	The group $G_s$ acts on $Z_s=\cone(G_s)$ and $A_s=G_s$ is an invariant subspace.
	Let $G':=\prod_{s\in S} G_s$ denote the direct product.
	Then $G'$ acts on $\prod_{s\in S} Z_s$ and $Z':=\piprod_L (\cone (G_s), G_s)$ is an invariant subspace.
	The quotient space $Z'/G'$ can be identified with the chamber $K=\piprod_L ([0,1],0)$ of the previous example.
	It is proved in \cite{d09} that the universal cover of $Z'$ is the standard realization of a RAB. In particular, the group $G$ of all lifts of elements in $G'$ is the graph product, $G=\prod_\gG G_s$.
	An explanation for this, which is different from that in \cite{d09}, is given in the following lemma.
\end{example}
\begin{lemma}\label{l:cover}
	With notation as in Example~\ref{ex:cone}, the fundamental group of $Z'=\piprod_L(\cone(G_s), G_s)$ can be identified with the kernel of the natural surjection $G=\prod_\gG G_s \to G' =\prod_{s\in S} G_s$.
	Moreover, if $Z\to Z'$ is the corresponding covering space, then the $G'$-action on $Z'$ lifts to a $G$-action on $Z$.
\end{lemma}
\begin{proof}
	First consider the special case where $S$ consists of two elements $s$ and $t$ and $\gG$ has no edges.
	Then $\cone (G_{s})\times \cone (G_{t})$ is a $2$-complex and the polyhedral product $Z'$ is the union of $1$-cells which do not contain the product of the cone points.
	Such a $1$-cell either has the form $\cone (g_s)\times g_t$ or $g_s\times \cone (g_t)$ for some $(g_s, g_t)\in G_s\times G_t$.
	These two $1$-cells fit together to give a single edge $e(g_s,g_t):=(\cone (g_s)\times g_t) \cup (g_s \times\cone(g_t))$ connecting $g_t$ to $g_s$.
	In this way we see that $Z'$ is identified with (the barycentric subdivision of) the join of $G_s$ and $G_t$, which we denote $G_s\odot G_t$.
	The group $G_s\times G_t$ acts on $G_s\odot G_t$ and the vertex stabilizers are either $G_s$ or $G_t$.
	The universal cover of $G_s\odot G_t$ is a tree $T$.
	The group of all lifts of the $(G_s\times G_t)$-action is transitive on edges, and the quotient space is a single edge (with distinct vertices).
	Hence, the group of lifts is the free product $G_s*G_t$ and $T$ is the corresponding Bass--Serre tree.

	The general case follows in the same manner by considering the universal cover of the $2$-skeleton of $Z'$.
\end{proof}

Next suppose that for each $s\in S$ we are given a path connected $G_s$-space $Z_s$ and a basepoint $b_s\in Z_s$ lying in some free orbit.
We can find a $G_s$-equivariant map of pairs $f_s:(Z_s, G_sb_s)\to (\cone (G_s), G_s)$.
We want to define a space $\prod_\gG (Z_s, G_s b_s)$, together with a $G$-action on it (where $G:=\prod_\gG G_s$).
It will be called the \emph{graph product} of the $(Z_s,G_sb_s)$.
The $f_s$ induce a map, well-defined up to $G'$-equivariant homotopy, $f:\piprod_L(Z_s, G_sb_s) \to \piprod_L (\cone (G_s), G_s)$.
It is easy to see that $f$ induces a surjection on fundamental groups.
Pulling back the universal cover of $\piprod_L (\cone(G_s), G_s)$, we get a covering space, $Z\to \piprod_L (Z_s, G_sb_s)$.
We use the notation $\prod_\gG (Z_s, G_s b_s) :=Z$ for this covering space.
Notice that the $G'$-action on $\piprod_L (Z_s, G_sb_s)$ lifts to a $G$-action on $Z$.
Also, notice that if each $Z_s$ is simply connected, then $Z$ is just the universal cover of the polyhedral product.
\begin{example}\label{ex:spaces}
	Suppose $\{(B_s,p_s)\}_{s\in S}$ is a collection of path connected spaces with base points and that $B=\piprod_L (B_s,p_s)$ is the polyhedral product.
	For each $s\in S$, let $\pi_s:Z_s\to B_s$ be the universal cover and let $G_s=\pi_1(B_s, p_s)$.
	Let $Z'$ denote the polyhedral product $\piprod_L(Z_s,\pi_s^\minus (p_s))$.
	Then $Z'\to B$ is a regular covering space with group of deck transformations $G'$ (the product of the $G_s$).
	It follows that the universal cover of $Z'$ can be identified with the graph product of the $(Z_s,\pi_s^\minus(p_s))$.
	Hence, $\pi_1(B)$ is the graph product of the $G_s$.
\end{example}
\begin{example}\label{ex:EG}
	Suppose $Z_s=EG_s$, the universal cover of the classifying space $BG_s$.
	A simple argument using induction on the number of elements of $S$ (cf.~\cite[Remark p.~619]{cdjams}) shows that the polyhedral product of the $(BG_s,p_s)$ is aspherical; hence, it is a model for $BG$ and its universal cover $\prod_\gG (EG_s, G_sb_s)$ is $EG$.
\end{example}
\begin{lemma}\label{l:spaces}
	(i) If each $G_s$ acts properly on $Z_s$, then $G$ acts properly on the graph product $\prod_\gG(Z_s,G_sb_s)$.

	(ii) If each $Z_s$ is acyclic, then so is $\prod_\gG(Z_s,G_sb_s)$.
\end{lemma}
\begin{proof}
	The proof of (i) is trivial.
	To prove (ii), consider the cover of $Z$ ($:=\prod_\gG(Z_s,G_sb_s)$) by components of the inverse images of the $\{Z_J\}_{J\in \cs(L)}$.
	By the K\"unneth Formula, each $Z_J$ is acyclic and the same is true for each component of its inverse image (since such a component projects homeomorphically).
	There is a similar cover of $\prod_\gG (\cone (G_s), G_s)$ with the same nerve.
	So, $Z$ and $\prod_\gG (\cone (G_s), G_s)$ have the same homology.
	Since $\prod_\gG (\cone (G_s), G_s)$ is the standard realization of a building, it is contractible; hence, acyclic.
	Statement (ii) follows.
\end{proof}
\begin{Remark}
	Probably the correct level of generality at which to define the graph product of a family of pairs of spaces is the following.
	Suppose we are given a family of pairs $\{(Z_s,A_s)\}_{s\in S}$, where each $Z_s$ is path connected and where $A_s$ is a not necessarily connected, closed nonempty subspace.
	Then $\prod_\gG (Z_s, A_s)$ can be defined in the following manner.
	First notice that Example~\ref{ex:cone} works when the discrete groups $G_s$ are replaced by discrete spaces $D_s$.
	Let $H$ denote the fundamental group of $\piprod_L(\cone (D_s), D_s)$.
	If $D_s=\pi_0(A_s)$ denotes the set of components of $A_s$, then, as before, we have maps $f_s:(Z_s,A_s) \to (\cone (D_s),D_s)$.
	The $f_s$ induce a map $f:\piprod_L (Z_s, A_s) \to \piprod_L (\cone (D_s), D_s)$.
	Moreover, the induced map of fundamental groups $f_*:\pi_1(\piprod_L (Z_s, A_s))\to H$ is surjective.
	The corresponding covering space $Z$ is called the \emph{graph product} of the $(Z_s, A_s)$ and is denoted by $\prod_\gG (Z_s,A_s)$.
	(Notice that if each $A_s$ is connected, then the graph product is the polyhedral product, $\piprod_L (Z_s,A_s)$.) In particular, this allows us to deal with the case where each $A_s$ is a $G_s$-orbit (not necessarily a free $G_s$-orbit).
	So, suppose $A_s= G_s/H_s$.
	Then the group of lifts of the $G'$-action to the universal cover of $\piprod_L(\cone(G_s/H_s), G_s/H_s)$ is the ``graph product of pairs,'' $\prod_\gG (G_s, H_s)$, defined previously in \cite{js}.
\end{Remark}

\section{A spectral sequence}\label{s:ss}
A \emph{poset of coefficients} is a contravariant functor $\ca$ from a poset $\cp$ to the category of abelian groups (i.e., it is a collection $\{\ca_a\}_{a\in \cp}$ of abelian groups together with homomorphisms $\gf_{ba}:\ca_a\to \ca_b$, defined whenever $a>b$, such that $\gf_{ca}=\gf_{cb}\,\gf_{ba}$, whenever $a>b>c$).
The functor $\ca$ also gives us a system of coefficients on the cell complex $\flag (\cp)$: it associates to the simplex $\gs$ the abelian group $\ca_{\min(\gs)}$.
Hence, we get a cochain complex
\begin{equation}\label{e:posetcochain}
	C^j(\flag(\cp);\ca):=\bigoplus_{\gs\in \flag(\cp)^{(j)}} \ca_{\min(\gs)},
\end{equation}
where $\flag(\cp)^{(j)}$ means the set of $j$-simplices in $\flag(\cp)$.

Given a CW complex $Y$, a \emph{poset of spaces} in $Y$ over $\cp$ is a cover $\cv=\{Y_a\}_{a\in \cp}$ of $Y$ by subcomplexes indexed by $\cp$ so that if $N(\cv)$ denotes the nerve of the cover, then
\begin{enumeratei}
	\item $a<b$ $\implies$ $Y_a \subset Y_b$, and
	\item the vertex set $\vertex(\gs)$ of each simplex of $N(\cv)$ has the greatest lower bound $\wedge \gs$ in $\cp$, and
	\item $\cv$ is closed under taking finite nonempty intersections, i.e., for any simplex $\gs$ of $N(\cv)$,
	\[
		\bigcap_{a\in \gs} Y_a= Y_{\wedge\gs}.
	\]
\end{enumeratei}
\begin{Remark}
	Any cover leads to a poset of spaces by taking all nonempty intersections as elements of new cover, and removing duplicates.
	The resulting poset is the \emph{set} of all nonempty intersections, ordered by inclusion.
\end{Remark}
\begin{lemma}[cf.~\cite{haefliger92}]\label{l:ss}
	Suppose $\cv=\{Y_a\}_{a\in \cp}$ is a poset of spaces for $Y$ over $\cp$.
	There is a Mayer--Vietoris type spectral sequence converging to $H^*(Y)$ with $E _1$-term:
	\[
		E^{i,j}_1= C^i(\flag(\cp); \ch^j(\cv)),
	\]
	and $E_2$-term:
	\[
		E^{i,j}_2= H^i(\flag(\cp); \ch^j(\cv)),
	\]
	where the coefficient system $\ch^j(\cv)$ is given by $\ch^j(\cv)(\gs)=H^j(Y_{\min(\gs)})$.
\end{lemma}
\begin{proof}
	We follow the line laid down in \cite[Ch.~VII, \S 3,4]{brown}.
	Consider the following double complex:
	\begin{equation}\label{e:e0}
		E_0^{i,j}=\bigoplus_{\substack{\gs \in \flag(\cp)\\
		\dim \gs =i}} C^j(Y_{\min(\gs)}),
	\end{equation}
	where the differentials are defined as follows.
	The vertical differentials are direct sums of the differentials $d: C^j(Y_{\min(\gs)}) \to C^{j+1}(Y_{\min(\gs)})$.
	Similarly, the horizontal differentials are direct sums of homomorphisms $\gd: C^j(Y_{\min(\gs)}) \to C^j(Y_{\min(\gt)})$ where the matrix entry corresponding to $\gs \gt$ is $[\gt:\gs]i_{\gt\gs}$, where $[\gt:\gs]$ is the incidence number and $i_{\gt\gs}: C^j(Y_{\min(\gs)}) \to C^j(Y_{\min(\gt)})$ is the restriction map if $\gt$ is a coface of $\gs$ and $0$ otherwise.
	As in \cite[p.165]{brown}, there are two spectral sequences associated to the double complex.
	The first begins by taking vertical cohomology to get $E_1$ and then takes horizontal cohomology to get $E_2$.
	The differential on the $E_r$ sheet has bidegree $(r,-r+1)$.
	The second spectral sequence begins with the horizontal differential so that the differential on the $E_r$ sheet has bidegree $(-r+1,r)$

	The usual inclusion-exclusion argument using properties (i)--(iii) of a poset of spaces shows that the rows of the double complex are exact, except when $i=0$, where we get $C^j(Y)$ as the $E^{0,j}_1$-term of the second spectral sequence.
	This implies that the second spectral sequence collapses at $E_2$ and that the cohomology of the double complex is $H^*(Y)$ (cf.~the exercise in \cite[p.165]{brown}).

	We can rewrite \eqref{e:e0} as $E^{i,j}_0=C^i(\flag(\cp); \cac^j(\cv))$, where the coefficient systems are defined by $\cac^j(\cv)(\gs)=C^j(Y_{\min(\gs)})$.
	So, the first spectral sequence is the one claimed in the lemma.
\end{proof}

For $a\in \cp$, let $Y_{<a}:= \bigcup_{b<a} Y_b$.
For any maximal element $a \in \cp$, put $Y_{\neq a}:= \bigcup_{b\neq a} Y_b$.
Consider the following two conditions on the poset of spaces.
\begin{itemize}
	\item[(Z$'$)] For any $a,b\in \cp$ with $b<a$, the induced homomorphism $H^*(Y_a)\to H^*(Y_b)$ is the zero map.
	\item[(Z)] For any $a\in \cp$, the induced homomorphism $H^*(Y_a) \to H^*(Y_{<a})$ is the zero map.
\end{itemize}
Note that (Z) implies (Z$'$) since the map $H^*(Y_a)\to H^*(Y_b)$ factors through $H^*(Y_{<a})$.
\begin{lemma}\label{l:main}
	Suppose $\cv=\{Y_a\}_{a\in \cp}$ is a poset of spaces for $Y$ over $\cp$.
	\begin{enumeratei}
		\item\label{l:mainii}
		If (Z\,$'$) holds, then
		\[
			E^{i,j}_2=\bigoplus_{a\in \cp} H^i(\flag(\cp_{\ge a}),\flag(\cp_{>a}); H^j(Y_a)).
		\]

		\item\label{l:mainiii}
		If (Z) holds, then the spectral sequence degenerates at $E_2$ and
		\[
			\grh^*(Y)=\bigoplus_{a\in \cp} H^i(\flag(\cp_{\ge a}), \flag(\cp_{> a}); H^j(Y_a)).
		\]
	\end{enumeratei}
\end{lemma}
\begin{proof}
	We use the double complex from the proof of Lemma~\ref{l:ss}.
	The cochain group decomposes as a direct sum:
	\[
		C^i(\flag(\cp); \cac^j(\cv))= \bigoplus_{a\in \cp} C^i(\flag(\cp_{\ge a}),\flag(\cp_{> a}); C^j(Y_a)).
	\]
	The vertical differentials at $E_0$ respect this decomposition, so at $E_1$ the spectral sequence always decomposes as a direct sum:
	\[
		E_1^{i,j}=\bigoplus_{a\in \cp} C^i(\flag(\cp_{\ge a}),\flag(\cp_{> a}); H^j(Y_a)).
	\]
	In general, the differentials at $E_1$ do not respect this decomposition; however, condition (Z$'$) implies that they do, and therefore, the spectral sequence also decomposes at $E_2$:
	\begin{equation}\label{e:e2}
		E_2^{i,j}=\bigoplus_{a\in \cp} H^i(\flag(\cp_{\ge a}),\flag(\cp_{> a}); H^j (Y_a)).
	\end{equation}
	This proves (\ref{l:mainii}).

	Now suppose (Z) holds.
	By induction, we can assume that (\ref{l:mainiii}) is true for all posets with fewer elements.
	If $z \in E_0^{i,j}$ is a vertical cocycle, then its higher differential is given by $d_r(z)=\gd(x_r)$, where $(x_0=z,x_1, \dots , x_r)$ is any sequence of elements satisfying $x_k \in E_0^{i+k,j-k}$ and $\gd(x_k)=d(x_{k+1})$.
	Since the columns of double complex split as direct sums over $\gs$, the vertical cocycles split as a direct sum, and it suffices to show that higher differentials vanish for each summand.
	So, let $\gs$ be a simplex in $\flag(\cp)$, and consider the term $C^j(Y_{\min(\gs)})$.
	There are two cases.

	1) $\gs$ is a face of a simplex $\gt$ with $\min(\gt)=\min(\gs)$.
	Then $i_{\gt\gs}$ is the identity map and this forces higher differentials to be trivial on this term.
	Indeed, if $z\in C^j(Y_{\min(\gs)})$, then $d(x_1)=\gd(z)$, and therefore $d(x_{1\gt})=\pm z$, where $x_{1\gt}$ denotes the $\gt$ component of $x_1$.
	Thus we can choose $x_1'=\pm\gd(i_{\gt\gs}^{-1}(x_{1\gt}))$ and all higher $x_k=0$.

	2) $\gs$ is a ``maximal'' face, i.e., all its cofaces have strictly smaller minimum vertices.
	Then $a=\max(\gs)$ is a maximal element of $\cp$.
	The cover of $Y_{\neq a}$ by $\{Y_b \mid b\neq a\}$ is a poset of spaces over $\cp_{\neq a}$.
	Let $E_{0,a}$ be the subcomplex of $E_0$ corresponding to the pair $(Y, Y_{\neq a})$:
	\[
		E_{0,a}^{i,j}=\bigoplus_{\substack{\max(\gt)=a \\
		\dim \gt =i}} C^j(Y_{\min(\gt)}).
	\]
	Note that $C^j(Y_\gs)$ is contained in this subcomplex, so it suffices to show that the higher differentials vanish for $E_{0,a}$.
	The pair $(Y, Y_{\neq a})$ excises to the pair $(Y_a, Y_{<a})$.
	So we have a short exact sequence:
	\[
		0 \to E_{0,a} \to E_{0,\le a} \to E_{0,< a} \to 0\,,
	\]
	where
	\[
		E_{0,\le a}^{i,j}=\bigoplus_{\substack{\gt \in \flag(\cp_{\le a})\\
		\dim \gt =i}} C^j(Y_{\min(\gt)}),
	\]
	and
	\[
		E_{0,< a}^{i,j}=\bigoplus_{\substack{\gt \in \flag(\cp_{< a})\\
		\dim \gt =i}} C^j(Y_{\min(\gt)}).
	\]
	are double complexes corresponding to the covers (posets of spaces) of $Y_a$ by $\{Y_b \mid b\le a\}$ and of $Y_{<a}$ by $\{Y_b \mid b< a\}$.

	The $E_2$ terms of the spectral sequences $E_{\le a}$ and $E_{< a}$ are
	\[
		E_{2, \le a}^{i,j}=\bigoplus_{b\le a} H^i(\flag(\cp{[b,a]}),\flag(\cp{(b,a]}); H^j(Y_b)),
	\]
	and
	\[
		E_{2, < a}^{i,j}=\bigoplus_{b< a} H^i(\flag(\cp{[b,a)}),\flag(\cp{(b,a)}); H^j(Y_b)).
	\]
	For $b<a$, $\flag([b,a])$ is a cone on $\flag((b,a])$; therefore, the only nonzero terms in $E_{2, \le a}$ come from $b=a$ and $i=0$, i.e., $E_{2, \le a}$ has $H^j(Y)$ in the 0-row and 0 everywhere else.
	In particular, it collapses at $E_2$.
	By inductive assumption $E_{2, < a}$ also collapses at $E_2$.
	Since, by hypothesis, $H^*(Y_a)\to H^*(Y_{<a})$ is the 0-map, the long exact sequence of the pair $(Y_a, Y_{<a})$ splits into short exact sequences, and similarly for the $E_2$ terms, it follows that the spectral sequence $E_a$ also collapses at $E_2$ term.
	Thus, the higher differentials in $E$ are 0.
\end{proof}

\section{Some previous cohomology computations}\label{s:cohom}
Suppose $G$ is a discrete group and $Y$ is a $G$-CW complex.
Let $N$ be a left $G$-module.
The \emph{$G$-equivariant cochain complex} is defined by,
\begin{equation*}
	C^i_G(Y;N):= \Hom_G(C_i(Y);N),
\end{equation*}
where $C_*(Y)$ denotes the usual cellular chain complex.
Its cohomology is denoted $H^*_G(Y;N)$.
If $G$ acts freely on $Y$, then $C^*_G(Y;N)$ can be identified with $C^*(Y/G;N)$, the cellular cochains on the quotient space with local coefficients in $N$.
There is a similar result even when the action is not free; however, the coefficients will no longer be locally constant.
Rather, the coefficients will be in a contravariant functor $\ci(N)$ from the poset of cells in $Y/G$ to the category of abelian groups: $\ci(N)$ assigns to a cell $c$ the fixed submodule $N^{G_c}$, where $G_c$ denotes the stabilizer of some lift of $c$ and where
\begin{equation}\label{e:coeff}
	C^i_G(Y)=C^i(Y/G;\ci(N)).
\end{equation}
For $Y=EG$, the universal cover of the classifying space $BG$, define the \emph{cochains and cohomology of $G$ with coefficients in $N$} by
\begin{align*}
	C^*(G;N)&:=C^*_G(EG;N)=C^*(BG;N),\\
	H^*(G;N)&:=H^*_G(EG;N)=H^*(BG;N).
\end{align*}

The spectral sequence arguments from \S\ref{s:ss} work with equivariant cochains (in particular with cochains with local coefficients) as long as the cover $\cv$ is $G$-equivariant.

In what follows we will be interested principally in two cases: $N=\zz G$, the group ring, and $N=\cN(G)$, the group von Neumann algebra.
We recall the definitions and some results which have been proved previously.

\paragraph{Group ring coefficients.} In the case of group ring coefficients, if $G$ acts properly and cocompactly on $Y$, there is the following formula (cf.~\cite[Ex.~4, p.~209]{brown}),
\begin{equation*}\label{e:hc}
	H^*_G(Y;\zz G)=H^*_c(Y),
\end{equation*}
where $H^*_c(Y)$ means cohomology with compact supports.
Thus, Lemma~\ref{l:spaces} implies the following.
\begin{corollary}\label{c:gprdt}
	For each $s\in S$, suppose $G_s$ is a discrete group and that $(Z_s,G_sb_s)$ a $G_s$-CW complex together with a free orbit.
	Also suppose each $G_s$-action is proper and cocompact and that $Z_s$ is acyclic.
	Then for $G=\prod_\gG G_s$ and $Z=\prod_\gG (Z_s,G_sb_s)$, we have
	\[
		H^*(G;\zz G)=H^*_c(Z).
	\]
\end{corollary}

The cohomology of Coxeter groups are computed as follows.
\begin{theorem}[\cite{d98} as well as \cite{ddjo2}]\label{t:d98}
	\[
		H^n(W;\zz W)=H^n_c(\cu(W,K))=\bigoplus_{J\in\cs(W,S)} H^n(K,K^{S-J})\otimes \h{A}(W)^J,
	\]
	where $\h{A}(W)^J$ is the free abelian group on the set of $w\in W$ which have reduced expressions ending with letters in $J$.
\end{theorem}
Using this, Jensen and Meier proved the following.
(A different proof of this will be given in \S\ref{ss:artin2}.)
\begin{theorem}[Jensen--Meier \cite{jm}]\label{t:jm}
	Suppose $(W,S)$ is a RACS and $A_L$ is the associated RAAG. Then
	\[
		H^n(A_L;\zz A_L)=\bigoplus_{J\in \cs(W,S)} H^{n-|J|}(K_J,\partial K_J) \otimes \zz (A_L/A_J).
	\]
\end{theorem}
\begin{theorem}\label{t:ddjmo}
	\textup{(\cite[Corollary 8.2]{ddjmo}).} Suppose $\cac$ is a locally finite building of type $(W,S)$.
	Then
	\[
		H^n_c(\cb(W,K))=\bigoplus_{J\in\cs(W,S)} H^n(K,K^{S-J})\otimes \h{A}(\cac)^J,
	\]
	where $\h{A}(\cac)^J$ is a certain subgroup of the free abelian group on $\cac$.
\end{theorem}

In particular, by Lemma~\ref{l:RAB}, this theorem gives the following computation of \cite{ddjo2} of the compactly supported cohomology of a locally finite RAB and hence, of the cohomology with group ring coefficients of the graph product of a collection of finite groups.
\begin{theorem}[{\cite[Theorem 6.6]{ddjo2}}]\label{t:ddjo2}
	Suppose $\{G_s\}_{s\in S}$ is a collection of finite groups and that $G=\prod_\gG G_s$.
	Then
	\[
		H^n(G;\zz G)=\bigoplus_{J\in \cs(L)} H^n(K,K^{S-J})\otimes \hat{A}(J),
	\]
	where $\hat{A}(J)$ is a certain (free abelian) subgroup of $\zz (G/G_J)$.
\end{theorem}

\paragraph{$L^2$-cohomology and $L^2$-Betti numbers.} The real \emph{group algebra}, $\bR G$, of $G$ consists of all finitely supported functions $G\to \bR$.
Its \emph{standard basis} is $\{e_g\}_{g\in G}$, where $e_g$ denotes the indicator function of $\{g\}$.
The \emph{standard inner product} on $\bR G$ is defined by $e_g\cdot e_h=\gd_{gh}$, where $\gd_{gh}$ is the Kronecker delta.
The Hilbert space completion of $\bR G$, denoted $L^2(G)$, consists of all square summable functions $G\to \bR$.
The group $G$ acts orthogonally on $L^2(G)$ by either left or right translation.
To fix ideas, let us say that it is the right action defined by left translation.
The \emph{von Neumann algebra} of $G$, denoted by $\cN(G)$, is the commutant of the $G$-action.
It acts on $L^2(G)$ from the left.
For $\gf\in \cN(G)$, define
\begin{equation*}
	\tr_{\cN(G)}(\gf):= (\gf e_1)\cdot (e_1).
\end{equation*}
If $V$ is a closed $G$-stable subspace of a finite direct sum of copies of $L^2(G)$, then its \emph{von Neumann dimension} is defined by
\begin{equation*}
	\dim_{\cN(G)} V:=\tr_{\cN(G)} (p_V),
\end{equation*}
where $p_V:\oplus L^2(G)\to \oplus L^2(G)$ is orthogonal projection onto $V$.

Suppose the $G$-CW complex $Y$ is proper and cocompact.
Define $L^2C^*(Y)$ to be the cochain complex of real-valued, square summable cochains on $Y$.
Denote its reduced cohomology group by $L^2H^*(Y)$.
(Here ``reduced'' means $\Ker \gd/ \ol{\Ima \gd}$, where $\gd:L^2C^i(Y)\to L^2C^{i+1}(Y)$ is the coboundary operator.
It is necessary to take the closure of $\Ima \gd$ for the quotient to be a Hilbert space.) Define the $i^{th}$ $L^2$-Betti number by
\[
	L^2b^i(Y;G):=\dim_{\cN(G)} L^2H^i(Y).
\]
If $Y$ is acyclic, then $L^2H^i(Y)$ depends only on $G$ and is denoted by $L^2H^i(G)$ and similarly, $L^2b^i(G):= L^2b^i(Y;G)$.
Thus, Lemma \ref{l:spaces} implies the following.
\begin{corollary}\label{c:gprdtb}
	For each $s\in S$, suppose $G_s$ is a discrete group and that $(Z_s,G_sb_s)$ a $G_s$-CW complex together with a free orbit.
	Also suppose each $G_s$-action is proper and cocompact and that $Z_s$ is acyclic.
	Then for $G=\prod_\gG G_s$ and $Z=\prod_\gG (Z_s,G_sb_s)$, we have
	\[
		L^2b^*(G)=L^2b^i(Z;G).
	\]
\end{corollary}
The $L^2$-Betti numbers of Coxeter groups have proved to be difficult to compute.
Some partial results and conjectures can be found in \cite{do}.
For locally finite buildings of very large thickness there is a complete calculation due to Dymara--Januszkiewicz \cite{dj}.
The requirement on the thickness is reduced in \cite{ddjo07} (cf.~Theorems~\ref{t:ddjo} and \ref{t:ddjobldg} in \S\ref{ss:betti}).

In the case of Artin groups, we have the following easy computation of \cite{dl}.
(A proof of this will be given in the next section.)
\begin{theorem}[\cite{dl}]\label{t:dl}
	\( L^2 b^n (\wt{X};A) = b^n (K,\partial K)=\ol {b}^{n-1}(L), \) where, as usual, $b^n (K,\partial K)=\dim (H^n(K,\partial K;\bR))$ and $\ol {b}^{n-1}(L)=\dim (\rh^{n-1}(L;\bR))$.
\end{theorem}

In \cite{luck}, L\"uck shows that there is an equivalence of categories between the category of finitely generated $\cN(G)$-modules and the category of orthogonal representations of $G$ on Hilbert spaces which are $G$-isomorphic to closed, $G$-stable subspaces of a finite direct sum of copies of $L^2(G)$.
Given a finitely generated $\cN(G)$-module $E$, define $\dim_{\cN(G)}E$ to be the von Neumann dimension of the corresponding Hilbert space.
Then
\[
	L^2b^i(Y;G)= \dim_{\cN (G)}H^i_G(Y;\cN(G)).
\]
Just as in \eqref{e:coeff}, we have that
\begin{equation}\label{e:coeff2}
	H^*_G(Y;\cN(G))=H^*\bigl (Y/G;\ci(\cN(G))\bigr ).
\end{equation}

\section{Computations}\label{s:comp}

\subsection{Artin groups}\label{ss:artin2}
As in \S\ref{ss:real}, $A$ is the Artin group associated to a Coxeter system $(W,S)$ and $X$ is its Salvetti complex.
As usual, $L=L(W,S)$, $\cs=\cs(W,S)$ and $K:=\flag(\cs)$.
We wish to compute $H^*(X;\zz A)$.
Given a spherical subset $J\in \cs$, $A_J$ is the corresponding Artin group and $X_J$ is its Salvetti complex.
We know that $X_J$ is the classifying space for $A_J$.
By \cite{squier} (see also \cite{bestvina}), for each spherical subset $J$, $A_J$ is a duality group of dimension $|J|$.
This means that $H^*(A_J;\zz A_J)$ is zero for $*\neq |J|$ and that $F_J:=H^{|J|}(A_J;\zz A_J)$ is free abelian.

As explained in \S\ref{ss:real}, the cover $\cv=\{X_J\}_{J\in \cs}$ is a poset of spaces for $X$.
In the case of group ring coefficients, we have a spectral sequence of the type considered in \S\ref{s:ss} converging to $H^*(X;\zz A)$.
It has $E_2$ term: $E_2^{i,j}=H^i(K;\ch^j(\cv))$, where $\ch^j(\cv)$ is the coefficient system, $\gs \mapsto H^i(X_{\min \gs};\zz A)$,. By Lemma~\ref{l:main}, once we establish condition (Z) of \S\ref{s:ss} we will get the following calculation.
\begin{theorem}\label{t:salvetti}
	\[
		\grh^n(X;\zz A)=\bigoplus _{J\in \cs(W,S)} H^{n-|J|}(K_J,\partial K_J) \otimes H^{|J|}(A_J;\zz A).
	\]
\end{theorem}

A similar argument can be used recover the calculation of the $L^2$-Betti numbers of $X$ in \cite{dl}.
(This computation was stated earlier as Theorem~\ref{t:dl}.) The spectral sequence has $E_2^{i,j}=H^i(K;\ch^j(\cv))$, where $\ch^j(\cv)$ is the coefficient system $\gs \mapsto H^i(X_{\min \gs};\cN(A))$.
The key observation in \cite{dl} for proving Theorem~\ref{t:dl} was that for $J \neq \emptyset$, all $L^2$-Betti numbers of $A_J$ vanish.
In particular, condition (Z) of \S\ref{s:ss} holds (since all cohomology groups vanish except when $J=\emptyset$).
So, $E^{i,j}_1$ is $0$ for $j\neq 0$ while
\[
	E^{i,0}_1= C^i(K,\partial K;\cN(A)),
\]
where the coefficients are now constant.
It follows that $H^n(X;\cN(A))\cong H^n(K,\partial K)\otimes \cN(A)$; whence, Theorem~\ref{t:dl}.

In the case of Theorem~\ref{t:salvetti}, Condition (Z) is basically the following lemma.
\begin{lemma}\label{l:artinZ}
	For any $J\in \cs$, $H^*(X_J;\zz A) $ is concentrated in degree $|J|$, where it is equal to the free abelian group $F_J\otimes _{A_J} \zz A$.
	Hence, $H^*(X_J;\zz A) \to H^*(X_{<J};\zz A)$ is the zero map.
\end{lemma}
\begin{proof}
	The first sentence is from \cite{squier}.
	The second sentence follows since $X_J$ is a $|J|$-dimensional CW complex, $H^*(X_J;\zz A_J)$ is concentrated in the top dimension and $X_{<J}$ is a subcomplex of one less dimension.
\end{proof}

Theorem~\ref{t:salvetti} follows immediately from Lemma~\ref{l:main}.
We note that if the $K(\pi,1)$ Conjecture holds for $A$ (i.e, if $X=BA$), then the formula in Theorem~\ref{t:salvetti} is a calculation of $H^*(A;\zz A)$ and Theorem~\ref{t:dl} gives a formula for $L^2b^n(A)$.
In particular, since the $K(\pi,1)$ Conjecture holds for RAAG's, Theorem~\ref{t:salvetti} gives as a corollary, a different proof the Jensen--Meier calculation in \cite{jm} (stated previously as Theorem~\ref{t:jm}).

\subsection{Bestvina--Brady groups}\label{ss:bb}
In this subsection $A_L$ is a RAAG, $T_L$ is its Salvetti complex defined in \eqref{e:tl} and $\pi:A_L\to \zz$ is the standard homomorphism.
We have a $\pi$-equivariant map $p:\wt{T}_L\to \bR$ and $BB_L=\Ker \pi$.
Put $Z_L=p^\minus(t)$ for some $t\in \bR-\zz$ (say for $t=\frac{1}{2}$).
It is proved in \cite{bb} that if $L$ is acyclic, then so is $Z_L$.
If this is the case, $BB_L$ is called a \emph{Bestvina--Brady group}.
We can compute equivariant cohomology of $Z_L$ by the method used in the proof of Theorem.~\ref{t:salvetti}.
\begin{theorem}\label{t:bb}
	The compactly supported cohomology of $Z_L$ is isomorphic to that of $\wt{T}_L$ shifted up in degree by $1$, i.e.,
	\begin{align*}
		H^n_c(Z_L)&= H^n_{BB_L}(Z_L;\zz BB_L)\\
		&=\bigoplus_{J\in \cs(L)_{>\emptyset}} H^{n-|J|+1}(K_J,\partial K_J)\otimes \zz (BB_L/(BB_L\cap A_J)).
	\end{align*}
	When $L$ is acyclic, this is a calculation $H^j(BB_L;\zz BB_L)$.
\end{theorem}
\begin{proof}
	We intersect the cover $\{\wt{T}_J\}_{J\in \cs(L)}$ of $\wt{T}_L$ with $Z_L$.
	Put $Z_J:=Z_L \cap \wt{T}_J$.
	Since $t$ is not an integer, $Z_L$ does not contain any vertices of the cubical complex $\wt{T}_L$, i.e., $Z_L\cap \wt{T}^{\,\emptyset} =\emptyset$.
	On the other hand, when $J$ is nonempty, the intersection of $Z_L$ with any component of $\wt{T}_J$ is a Euclidean subspace of codimension one and the collection of such intersections is in one-to-one correspondence with the cosets of $BB_L\cap A_J$ in $BB_L$.
	So, $\{Z_J\}_{J\in \cs(L)_{>\emptyset}}$ is a poset of spaces on $Z_L$.
	The simplicial complex $\flag(\cs(L)_{>\emptyset})$ is equal to $\partial K$ (i.e., the barycentric subdivision of $L$).
	The $E^{i,j}_1$-term of the spectral sequence is $C^i(K;\ch^j(\cv))$, where the coefficient system takes $\gs$ to $H^j_{BB_L}(Z_{\min \gs};\zz BB_L)$.
	Since each component of $Z_J$ is Euclidean space of dimension $|J|-1$, the coefficients, $H^j_{BB_L}(Z_J;\zz BB_L)$ are $0$ whenever $j\neq |J|- 1$.
	Moreover, for $j=|J|-1$,
	\[
		H^j_{BB_L}(Z_J;\zz BB_L)=H^j_c(\bR^j)\otimes_{BB_L \cap A_J} \zz BB_L=\zz(BB_L/(BB_L\cap A_J)).
	\]
	It follows that conditions (Z$'$) and (Z) of \S\ref{s:ss} hold (because $Z_{<J}$ is a subcomplex of dimension $|J|-2$).
	Hence, by Lemma~\ref{l:main}, the spectral sequence degenerates at $E_2$ and $E_2^{i,j}=\bigoplus_J E_{2,J}^{i,j}$, where $E_{2,J}^{i,j}$ is nonzero only for $j=|J|-1$, in which case,
	\[
		E^{i,|J|-1}_{2,J}= H^i(K_J,\partial K_J)\otimes \zz(BB_L/(BB_L\cap A_J)).
	\]
	The theorem follows.
	(We also note that for $J\neq \emptyset$, $\pi:A_J=\zz^J\to \zz$ is onto, so that $BB_L/(BB_L\cap A_J)\cong A_L/A_J$.)
\end{proof}

Similarly, we compute the $L^2$-Betti numbers of $Z_L$ as follows.
\begin{theorem}\label{t:lraag}
	Suppose $Z_L$ is a generic level set of the function $p:\wt{T}_L\to \bR$.
	Then
	\[
		L^2 b^n(Z_L;BB_L)=\sum_{s\in S} b^n(K_s,\partial K_s)= \sum_{s\in S} \ol{b}{}^{n-1} (\Lk(s)).
	\]
	In particular, when $L$ is acyclic,
	\[
		L^2b^n(BB_L)=\sum_{s\in S} b^n(K_s,\partial K_s).
	\]
\end{theorem}
\begin{proof}
	As before, the $E^{i,j}_1$-term of the spectral sequence is $C^i(K;\ch^j(\cv))$, where the coefficient system takes $\gs$ to $H^j(Z_{\min \gs};\cN(BB_L))$.
	Since each component of $Z_J$ is Euclidean space of dimension $|J|-1$, the coefficients, $H^*(Z_J;\cN(BB_L))$ are $0$ whenever $|J|\neq 1$.
	Hence, by Lemma~\ref{l:main}, the spectral sequence degenerates at $E_2$ and only $E_2^{i,0}$ can be nonzero, where
	\[
		E^{i,0}_2=\bigoplus_{s\in S} H^i(K_s,\partial K_s;\cN(BB_L)).
	\]
	The theorem follows.
\end{proof}
\begin{Remark}
	Here is a different proof of Theorem~\ref{t:bb} when $L$ is acyclic.
	Put $Y_+:=p^\minus([\frac{1}{2},\infty))$ and $Y_-:=p^\minus((\-\infty,\frac{1}{2}])$.
	We first claim that the compactly supported cohomology of $Y_\pm$ vanishes in all degrees.
	It suffices to consider $Y_+$, the argument for $Y_-$ being similar.
	The arguments of \cite{bb} show that when $L$ is acyclic the inclusion of any level set $p^\minus(t)$ into a sublevel set $p^\minus([\frac{1}{2}, t])$ induces an isomorphism on homology.
	The same argument shows that it induces an isomorphism on compactly supported cohomology, $H^*_c(p^\minus([\frac{1}{2}, t]))\to H^*_c(p^\minus(t))$.
	Hence, $H^*_c(p^\minus([\frac{1}{2}, t]), p^\minus(t))=0$.
	Since there is an excision, $H^*_c \bigl (Y_+,p^\minus([t,\infty))\bigr ) \cong H^*_c(p^\minus([\frac{1}{2}, t]), p^\minus(t))$, the left hand side also vanishes.
	For any compact subset $C\subset Y_+$ we have that $Y_+-C\supset p^\minus([t,\infty])$ for large enough $t$; so,
	\[
		H^*_c(Y_+)=\lim _{t\to \infty} H^*_c \bigl (Y_+, p^\minus ([t,\infty))\bigr ),
	\]
	and by the previous discussion the right hand side vanishes.
	Hence, so does $H^*_c(Y_+)$.
	We have $Y_+\cup Y_-=\wt{T}_L$ and $Y_+\cap Y_- = Z_L$ and a Mayer--Vietoris sequence:
	\[
		0= H^*_c(Y_+)\oplus H^*_c(Y_-)\to H^*(Z_L)\to H^{*+1}_c(\wt{T}_L)\to 0.
	\]
	The theorem follows from the computation of $H^{*+1}_c(\wt{T}_L)$ in Theorem~\ref{t:salvetti} (or in Theorem~\ref{t:jm}).
\end{Remark}
\begin{Remark}
	Theorem~\ref{t:jm} provides a calculation of $H^*(A_L;\zz A_L)$ as a sum of terms involving the $H^{*-|J|}(K_J,\partial K_J)$, where $\partial K_J\cong \Lk(J)$.
	In the calculation of $L^2$-cohomology in Theorem~\ref{t:dl} only the term with $J=\emptyset$ enters.
	Hence, under the canonical map, $H^*(A_L;\zz A_L)\to L^2H^*(A_L)$, all the terms with $J\neq \emptyset$ go to $0$.
	Similarly, in Theorem~\ref{t:bb} we calculated $H^*(BB_L;\zz BB_L)$ as a sum of terms involving $H^{*-|J|+1}(K_J, \partial K_J)$.
	(Since $L$ is acyclic, the term with $J=\emptyset$ does not appear.) On the other hand, in Theorem~\ref{t:lraag} for $L^2H^*(BB_L)$ only the terms with $|J|=1$ occur.
	So, the canonical map $H^*(BB_L;\zz BB_L)\to L^2H^*(BB_L)$ takes all the terms with $|J|>1$ to $0$.
\end{Remark}
\begin{Remark}
	The cohomology of $BB_L$ with trivial coefficients was computed by Leary and Saadeto\u{g}lu in \cite{is}.
\end{Remark}

\subsection{Graph products of infinite groups}\label{ss:cgfree}
As in the Introduction, $\{G_s\}_{s\in S}$ is a family of groups and $G=\prod_\gG G_s$ is the graph product with respect to the simplicial graph $\gG$.
The associated flag complex is $L$.
For each $J$ in $\cs(L)$, $G_J$ denotes the direct product of the $G_s$ with $s\in J$.
In this subsection we shall also suppose that \emph{each $G_s$ is infinite}.
Put $Y=EG$.
As in Example~\ref{ex:EG}, $EG$ is the graph product of the $(EG_s, G_sb_s)$.
The cover $\cv=\{Y_J\}_{J\in \cs(L)}$, where $Y_J=G\times _{G_J} EG_J$, is a poset of spaces structure for $EG$.
Let $N$ stand for $\zz G$ or $\cn(G)$.
The spectral sequence of \S\ref{s:ss} converges to $H^*(G;N)$ and has $E_2$ term:
\[
	E_2^{j.k}=H^i(K;\ch^j(\cv)),
\]
where the coefficient system is given by $\ch^j(\cv)(\gs)=H^j(G_{\min(\gs)};N)$.
Once we verify that Condition (Z) holds, Lemma~\ref{l:main}, will provide the following calculations.
\begin{theorem}\label{t:gpfree}
	Let $G$ be a graph product of groups $G_s$, each of which is infinite.
	Then
	\[
		\grh^n(G;\zz G)=\bigoplus_{\substack{J\in \cs(L)\\
		i+j=n}} H^i(K_J,\partial K_J; H^j(G_J;\zz G)).
	\]
\end{theorem}
\noindent (Note that $H^j(G_J;\zz G)= H^j(G_J;\zz G_J)\otimes _{G_J}\zz G$.)
\begin{theorem}\label{t:lgraph}
	Let $G$ be a graph product of groups $G_s$, each of which is infinite.
	Then
	\[
		L^2b^n(G)=\sum_{\substack{J\in \cs(L)\\
		i+j=n}} b^i(K_J,\partial K_J)\cdot L^2b^j(G_J).
	\]
\end{theorem}

To establish Theorem~\ref{t:gpfree} we need to verify conditions (Z$'$) and (Z) (or in fact just condition (Z)) which precedes Lemma~\ref{l:main}.
These conditions follow from statements (i) and (ii), respectively, in the next lemma.
\begin{lemma}\label{l:Kunneth}
	Suppose $G_J=\prod_{s\in J}G_s$ is the direct product of a collection of infinite groups indexed by a finite set $J$.
	Let $\{Y_s\}_{s\in J}$ be a collection of connected CW complexes with proper $G_s$-actions such that each $Y_s$ contains a free orbit, $G_sb_s$, and put $Y_J:=\prod_{s\in J} Y_s$.
	As in \S\ref{ss:spaces}, for each $I\subset J$, define
	\[
		Y'_I:=\prod_{s\in I} Y_s \times \prod_{s\in J-I} G_s b_s.
	\]
	Let $N$ stand for either $\zz G_J$ or $\cN(G_J)$.
	Then

	(i) The map induced by inclusion, $H^*_{G_J}(Y_J;N)\to H^*_{G_J}(Y'_I;N)$, is the zero map.

	(ii) More generally, if
	\[
		Y_{<J}:=\bigcup_{s\in J} Y'_{J-s},
	\]
	then the map induced by inclusion, $H^*_{G_J}(Y_J;N)\to H^*_{G_J}(Y_{<J};N)$, is the zero map.
\end{lemma}
\begin{proof}
	We shall prove this only in the case $N=\zz G_J$, the case $N=\cN(G_J)$ being entirely similar.
	The relative version of the K\"unneth Formula states that for pairs of spaces $(A,B)$ and $(A',B')$,
	\[
		H^n((A,B)\times (A',B'))=\bigoplus_{i+j=n} H^i(A,B;H^j(A',B')),
	\]
	(where $(A,B)\times (A',B')=(A\times A', (A\times B' )\cup (B\times A'))$ ).
	Similarly, if $(A,B)$ is a pair of $H$-spaces and $(A',B')$ a pair of $H'$-spaces, then for left $H$- and $H'$-modules $M$ and $M'$,
	\[
		H^n_{H\times H'}((A,B)\times (A',B');M\otimes M')= \bigoplus_{i+j=n} H^i_H(A,B;M\otimes H^j_{H'}(A',B';M')),
	\]
	By the exact sequence of the pair, showing $H^*_{G_J}(Y_J;\zz G_J) \to H^*_{G_J}(Y'_I;\zz G_J)$ is zero is equivalent to showing that
	\begin{equation}\label{e:pfK}
		H^*_{G_J}(Y_I\times (Y_{J-I},G_{J-I} b);M\otimes M')\to H^*_{G_J}(Y_I\times Y_{J-I};M \otimes M')
	\end{equation}
	is onto, where $M=\zz G_I$, $M'=\zz G_{J-I}$ and $b\in Y_{J-I}$ is a basepoint.
	If $J-I\neq \emptyset$, then since $G_{J-I}$ is infinite and acts properly, $Y_{J-I}$ is noncompact; hence, $H^0_{G_{J-I}}(Y_{J-I};M')=0$ and so \[H^j_{G_{J-I}}(Y_{J-I},G_{J-I}b;M')\to H^j_{G_{J-I}}(Y_{J-I};M')\] is onto.
	Hence, \[H^i_{G_I}(Y_I; M\otimes H^j_{G_{J-I}}(Y_{J-I},G_{J-I}b;M'))\to H^i_{G_I}(Y_I;M\otimes H^j_{G_{J-I}}(Y_{J-I};M'))\] is onto.
	It follows from the relative K\"unneth Formula that the map in \eqref{e:pfK} is onto.

	The proof of the second statement is similar using induction on the cardinality of $J$.
	Choose $s\in J$.
	Then $Y_{<J}= (Y_{<(J-s)} \times Y_s) \cup (Y_{J-s}\times G_s b_s)$.
	Hence,
	\[
		(Y_J,Y_{<J})=(Y_{J-s}, Y_{<(J-s)})\times (Y_s,G_s b_s).
	\]
	Let $M=\zz G_{J-s}$ and $M'=\zz G_s$.
	Then
	\[H^j_{G_{s}}(Y_s,G_sb_s;M')\to H^j_{G_{s}}(Y_s;M')\] is onto by the argument in the previous paragraph, and
	\[
		H^i_{G_{J-s}}(Y_{J-s}, Y_{<(J-s)};M)\to H^i_{G_{J-s}}(Y_{J-s};M)
	\]
	is onto by inductive hypothesis.
	Combining these two surjections, we see that
	\[
		H^i_{G_{J-s}}(Y_{J-s}, Y_{<(J-s)};M\otimes H^j_{G_s}(Y_s,G_s b_s;M'))\to H^i_{G_{J-s}}(Y_{J-s};M\otimes H^j_{G_{s}}(Y_s;M'))
	\]
	is onto.
	So, by the relative K\"unneth Formula,
	\[
		H^n_{G_{J-s}\times G_s}((Y_{J-s}, Y_{<(J-s)})\times (Y_s,G_s b_s);\zz G_J)\to H^n_{G_{J-s}\times G_s}(Y_{J-s}\times Y_s;\zz G_J)
	\]
	is onto, which completes the proof.
\end{proof}

\paragraph{Other coefficients.} As before, $G=\prod_\gG G_s$ is the graph product and $G'=\prod_{s\in S} G_s$ is the direct product.
Let $p:G\to G'$ be the natural projection.
We say that a group $H$ \emph{lies between $G$ and $G'$} if $H=G/N$ for some normal subgroup $N\subset G$ with $\Ker p \subseteq N$.
If this is the case, then $p$ factors as
\[
	G\,\mapright{f}\,H\,\mapright{}\, G'
\]
where $f$ is the natural epimorphism.
In this way $\zz H$ becomes a $G$-module.
There is the following generalization of Theorem~\ref{t:gpfree}.
\begin{theorem}\label{t:other}
	Suppose each $G_s$ is infinite and $H$ lies between $G$ and $G'$.
	Then
	\[
		\grh^n(G;\zz H)=\bigoplus_{\substack{J\in \cs(L)\\
		i+j=n}} H^i(K_J,\partial K_J; H^j(G_J;\zz H))
	\]
	where as before, $H^j(G_J;\zz H)=H^j(G_J;\zz G_J)\otimes _{G_J} \zz H$.
	Similarly,
	\[
		L^2b^n((EG)/N;H)=\sum_{\substack{J\in \cs(L)\\
		i+j=n}} b^i(K_J,\partial K_J)\cdot L^2b^j(G_J).
	\]
\end{theorem}

We have $H^n(G;\zz H)=H^n_H((EG)/N)$ and $(EG)/N$ is covered by complexes of the form $H\times_{G_J} EG_J$.
The proof then goes through in the same manner as that of Theorem~\ref{t:gpfree}.

\section{Graph products of Coxeter groups}\label{s:gpcg}
\subsection{Polyhedral joins}\label{ss:pj}
As in \S\ref{ss:poset} and \S\ref{ss:spaces}, let $L$ be a simplicial complex with vertex set $S$.
For each $s\in S$, suppose given a simplicial complex $\cl(s)$ with vertex set $T_s$.
For each $J\in \cs(L)$, define $\cl(J)$ to be the join,
\begin{equation}\label{e:join}
	\cl(J):=\Aster_{s\in J}\ \cl(s),
\end{equation}
and then define the \emph{polyhedral join} of the $\cl(s)$ with respect to $L$ by
\begin{equation}\label{e:pjoin}
	\aster_L\ \cl(s):=\bigcup_{J\in \cs(L)} \cl(J).
\end{equation}
To simplify notation put $\cl:=\aster_L\ \cl(s)$.
Here is an equivalent definition.
Let $T$ denote the disjoint union
\[
	T:=\bigcup_{s\in S} T_s.
\]
and $\pi:T\to S$ the natural projection.
Any subset $I$ of $T$ can be decomposed as
\[
	I=\bigcup_{s\in \pi(I)} I_s,
\]
where $I_s\subseteq T_s$.
Then $I$ is the vertex set of a simplex in $\cl$ if and only if \( \pi(I)\in \cs(L), \) and $I_s\in \cs(\cl(s))$ for each $s\in S$.
In other words, a simplex $I$ of $\cl$ is determined by a simplex $J\in \cs(L)$ and a collection of simplices $\{I_s\}_{s\in J}$, where each $I_s\in \cs(\cl(s))$.
\begin{Remark}
	Similarly, given any family of spaces $\{X(s)\}_{s\in S}$, for each $J\in \cs(L)$, define $X(J)$ to be the join of the $X(s)$ and the \emph{polyhedral join}, $\aster_L X(s)$, to be the union of the $X(J)$ as in \eqref{e:pjoin}.
\end{Remark}

Recall that the notion of ``polyhedral product'' was defined by \eqref{e:prod'}.
The proof of the next lemma is straightforward.
\begin{lemma}\label{l:commute}
	The operation of applying $K$ to simplicial complexes intertwines the polyhedral join with the polyhedral product (defined in \S\ref{ss:spaces}), i.e.,
	\[
		K(\aster_L\ \cl(s))=\piprod_L \bigl( K(\cl(s)),\partial K(\cl(s)) \bigr).
	\]
\end{lemma}
The following lemma is a well-known consequence of the K\"unneth Formula.
\begin{lemma}\label{l:join}
	Given two spaces $A$ and $B$,
	\[
		\rh^n(A\ast B)=\bigoplus_{i+j=n-1} \rh^i(A;\rh^j(B)).
	\]
	Also, $\rh^*(A\ast B)\to \rh^*(A)$ is the zero map whenever $B$ is nonempty.
	(Here $\rh^*(\ )$ means reduced cohomology.
	Also, we follow the convention that the reduced cohomology of the empty set is $\zz$ in degree $-1$.)
\end{lemma}
\begin{Remark}
	If we take coefficients in a field $\FF$, the formula in lemma~\ref{l:join} reads
	\[
		\rh^n(A\ast B;\FF)=\bigoplus_{i+j=n-1} \rh^i(A;\FF)\otimes \rh^j(B;\FF)
	\]
	Similarly, for the $J$-fold join, $X(J)$, of $\{X(s)\}_{s\in J}$,
	\[
		\rh^n(X(J);\FF)=\bigoplus_{\sum i_s=n-|J|-1}\ \bigotimes_{s\in J} \rh^{i_s}(X(s);\FF).
	\]
\end{Remark}
\begin{proof}[Proof of Lemma~\ref{l:join}] Write $CA$ and $CB$ for the cones on $A$ and $B$, respectively.
	Then $(CA,A)\times (CB,B)= (CA\times CB, A\ast B)$.
	Since $CA\times CB$ is contractible, $\rh^n(A\ast B)=H^{n+1}((CA,A)\times (CB,B))$.
	Hence,
	\begin{align*}
		\rh^n(A\ast B)&=H^{n+1}((CA,A)\times (CB,B))\\
		&=\bigoplus_{i+j+2=n+1} H^{i+1}(CA,A;H^{j+1}(CB,B))\\
		&=\bigoplus_{i+j=n-1} \rh^i(A;\rh^j(B)),
	\end{align*}
	where the second equation is the relative K\"unneth Formula.
	This proves the first sentence.
	To prove the second, note that the subspace $A\ast\emptyset\subset A\ast B$ is homotopy equivalent to $A\times CB\subset A\ast B$.
	So, by the exact sequence of the pair, we need only show $H^*(A\ast B, A\times CB) \to \rh^*(A\ast B)$ is onto.
	We have
	\begin{align*}
		H^n(A\ast B, A\times CB)&=H^n(CA\times B, A\times B)=H^n((CA,A)\times B)\\
		&=\bigoplus_{i+j=n} H^i(CA,A;H^j(B)).
	\end{align*}
	Since the connecting homomorphism $\rh^j(B)\to H^{j+1}(CB,B)$ is an isomorphism, it follows that $H^n(A\ast B,A)\to \rh^n(A\ast B)$ is onto.
\end{proof}

For any $J\in \cs(L)$, put
\begin{equation*}
	\cl(<J):=\bigcup_{s\in J} \cl(J-s),
\end{equation*}
where $\cl(J-s)$ is defined by \eqref{e:join}.
\begin{lemma}[cf.~Lemma~\ref{l:Kunneth}]\label{l:jKunneth}
	The map $\rh^*(\cl(J))\to \rh^*(\cl(<J))$, induced by the inclusion, is the zero homomorphism.
\end{lemma}
\begin{proof}
	The proof is by induction on the cardinality of $J$.
	It is trivially true for $|J|=1$.
	So assume $|J|>1$.
	We first claim that for each $s\in J$, $H^*(\cl(J),\cl(J-s))\to H^*(\cl(<J),\cl(J-s))$ is the zero map.
	We have $(\cl(J), \cl(J-s))=\cl(J-s)\ast(\cl(s),\emptyset)$.
	Hence, as in Lemma~\ref{l:join},
	\begin{equation}\label{e:K1}
		H^n(\cl(J),\cl(J-s))=\bigoplus_{i+j=n-1} \rh^i\bigl(\cl(J-s);H^j(\cl(s))\bigr).
	\end{equation}
	Similarly, $\cl(<J)=(\cl(<(J-s))\ast \cl(s)) \cup (\cl(J-s)\ast \emptyset)$; so,
	\begin{align}
		H^n(\cl(<J), \cl(J-s))&=H^n(\cl(<(J-s))\ast (\cl(s),\emptyset))\notag\\
		&=\bigoplus_{i+j=n-1} \rh^i\bigl(\cl(<(J-s)); H^j(\cl(s))\bigr).\label{e:K2}
	\end{align}
	By inductive hypothesis, $\rh^i(\cl(J-s))\to \rh^i\bigl(\cl(<(J-s))\bigr)$ is zero.
	Comparing \eqref{e:K1} and \eqref{e:K2}, we see that $H^*(\cl(J),\cl(J-s))\to H^*(\cl(<J),\cl(J-s))$ is the zero map, which proves the claim.
	By the exact sequence of the triple, this is equivalent to the statement that $H^*(\cl(J),\cl(<J))\to H^*(\cl(J),\cl(J-s))$ is onto.
	By Lemma~\ref{l:join}, $H^*(\cl(J),\cl(J-s))\to \rh^*(\cl(J))$ is also onto; hence, so is their composition, $H^*(\cl(J), \cl(<J))\to \rh^*(\cl(J))$.
	But this is equivalent to the statement that $\rh^*(\cl(J))\to \rh^*(\cl(<J))$ is zero, which is what we wanted to prove.
\end{proof}

Put $\ck:=K(\cl)$.
We want to use the spectral sequence of \S\ref{s:ss} to compute the cohomology of $(\ck,\ck^{T-I})$ for any $I\in \cs(\cl)$.
To warm up, let us do first the case $I=\emptyset$.
We note that $\ck^T$ is (the barycentric subdivision of) $\cl$ and $\ck$ is the cone on $\cl$.
Since
\[
	\ck=\piprod_L (K(\cl(s)), \cl(s))\quad \text{and} \quad \cl=\aster_L \cl(s),
\]
$(\ck,\cl)$ is a pair of posets of spaces over $\cs(L)$ (cf.~\eqref{e:join} and \eqref{e:pjoin}).
Lemma~\ref{l:jKunneth} says that condition (Z) of \S\ref{s:ss} holds; so Lemma~\ref{l:main} gives the following,
\begin{equation}\label{e:jcoh}
	\grh^n(\ck,\ck^T)=\bigoplus_{\substack{J\in \cs(L)\\
	i+j=n}} H^i\bigl(K_J,\partial K_J; H^j(\cone(\cl(J)), \cl(J))\bigr).
\end{equation}
Note that the term $H^j(\cone(\cl(J)), \cl(J))$ can be replaced by $\rh^{j-1}(\cl(J))$.
Let \( F:=\{s\in \cs(L)\mid \cl(s) \text{ is the simplex on }T_s\}.
\) Note that if $J\cap F\neq \emptyset$, then the join $\cl(J)$ is contractible (since one of its factors is a simplex).
So, in this case the coefficients in \eqref{e:jcoh}, $H^j(\cone(\cl(J)),\cl(J))$ vanish for all $j$.

Next, fix $I\in \cs(\cl)$.
For each $s\in S$, $I_s=I\cap T_s$.
Define a subset $G(I)$ of $S$ by
\begin{equation}\label{e:gi}
	G(I):=\{s\in \pi(I)\cap F\mid I_s=T_s\}.
\end{equation}
Then $G(I)$ is the vertex set of a simplex $\gs$ of $L$ (since $G(I)\subseteq \pi(I)$).
Let $^I\!L=L=\gs$ be the full subcomplex of $L$ spanned by $S-G(I)$, and let $^I\!K:=K(^I\!L)$ be the Davis chamber.
It is not hard to see that $^I\!L$ is homotopy equivalent to $K^{S-G(I)}$ (see \cite[Lemma~A.5.5, p.~416]{dbook}), where $\gs$ denotes the simplex corresponding to $G(I)$.
For each $s\in S-G(I)$, let $\cl^I(s)$ denote the full subcomplex of $\cl(s)$ spanned by $T_s-I_s$ and for each $J\in \cs(^I\!L)$, put
\begin{equation}\label{e:li}
	\cl^I(J):= \Aster_{s\in J} \cl^I(s).
\end{equation}
The usual spectral sequence argument proves the following.
\begin{theorem}\label{t:cohpjoin}
	With notation as above,
	\[
		\grh^n(\ck,\ck^{T-I})=\bigoplus_{\substack{J\in \cs(^I\!L)\\J\cap F=\emptyset\\
		i+j=n-1}} H^i\bigl({}^I\!K_J,\partial\, ^I\!K_J; \rh^j(\cl^I(J))\bigr).
	\]
\end{theorem}

There are two extreme cases of Theorem~\ref{t:cohpjoin}.
\begin{corollary}\label{c:pjoin1}
	Suppose each $\cl(s)$ is a simplex.
	Then
	\[
		H^n(\ck,\ck^{T-I})=H^n(^I\!K,\partial\, ^I\!K)
	\]
	($=\rh^{n-1}(K^{S-G(I)})=\rh^{n-1}(^I\!L$)).
\end{corollary}
\begin{proof}
	If each $\cl(s)$ is a simplex, then in Theorem~\ref{t:cohpjoin}, $J\cap F=J$ is nonempty unless $J= \emptyset$.
	When $J=\emptyset$, $\rh^j(\cl^I(J))$ is nonzero only for $j=-1$ and we get the formula in the corollary.
\end{proof}
\begin{corollary}\label{c:pjoin2}
	Suppose no $\cl(s)$ is a simplex.
	Then
	\[
		\grh^n(\ck,\ck^{T-I})=\bigoplus_{\substack{J\in \cs(L)\\
		i+j=n-1}} H^i\bigl(K_J,\partial K_J; \rh^j(\cl^I(J))\bigr).
	\]
\end{corollary}
\begin{proof}
	The hypothesis implies $G(I)=\emptyset$; hence, $^I\!L=L$ and $^I\!K=K$.
\end{proof}

\subsection{Cohomology of graph products of Coxeter groups with group ring coefficients}\label{ss:oct}
We continue from \S\ref{ss:spaces}.
We are given a family of Coxeter systems $\{(V_s,T_s)\}_{s\in S}$ and we form the graph product with respect to $\gG$.
Let $T$ denote the disjoint union of the $T_s$ and put $V=\prod_\gG V_s$.
Then $(V,T)$ is (obviously) a Coxeter system.
The projection $\pi:V\to W_L$ restricts to the natural projection $T\to S$ which sends $T_s$ to $s$.
For each $s\in S$ define $\cl(s)$ to be $L(V_s,T_s)$.
Clearly,
\begin{align}
	K(V,T)&=\piprod_{L(W,S)} K(V_s,T_s):= \ck\notag\\
	L(V,T)&=\aster _{L(W,S)} \cl(s):= \cl.\label{e:lvt}
\end{align}
Henceforth, we write $L$ and $K$ for $L(W,S)$ and $K(W,S)$, respectively.
\begin{notation}\label{n:vj}
	For each $J\in \cs(W,S)$, put $T(J)=\pim$ and
	\[
		V_J:=V_{\pi^\minus(J)}=\prod_{s\in J} V_s.
	\]
\end{notation}

We can combine Theorem~\ref{t:cohpjoin} with Theorem~\ref{t:d98} to get the following calculation of $H^*(V;\zz V)$.
\begin{theorem}\label{t:?}
	\[
		\grh^n(V;\zz V)=\bigoplus_{I\in \cs(V,T)}\ \bigoplus_{\substack{J\in \cs(^I\!L)\\J\cap F=\emptyset\\
		i+j=n}} H^i \big( {}^I\!K_J,\partial\, ^I\!K_J ;\rh^{j-1}(\cl^I(J))\big) \otimes \h{A}(V)^I.
	\]
\end{theorem}

On the other hand, in Theorem~\ref{t:gpfree} we calculated the cohomology of an arbitrary graph product of infinite groups and in Theorem~\ref{t:ddjo2} for an arbitrary graph product of finite groups.
We would like to see that these answers agree with the above in the case of Coxeter groups.

For any $J\in \cs(L)$, put
\[
	\ci(J):=\{I\in \cs(\cl)\mid G(I)=J\}.
\]
When each $V_s$ is finite,
\begin{align}
	\grh^n(V;\zz V)&= \bigoplus_{I\in \cs(\cl)} H^n(\ck,\ck^{T-I})\otimes \hat{A}(V)^I &\text{(by Theorem~\ref{t:d98} )}\notag\\
	&=\bigoplus_{J\in \cs(L)}\ \bigoplus_{I\in \ci(J)}\rh^{n-1}(K^{S-J}) \otimes \hat{A}(V)^I &\text{(by Corollary~\ref{c:pjoin1})}\notag\\
	&=\bigoplus_{J\in \cs(L)}H^n(K, K^{S-J})\otimes \bigoplus_{I\in \ci(J)}\hat{A}(V)^I,\label{e:finite}
\end{align}
where \eqref{e:finite} agrees with Theorem~\ref{t:ddjo2} with
\[
	\hat{A}(J):=\bigoplus_{I\in \ci(J)}\hat{A}(V)^I.
\]

Next consider the situation where all $V_s$ are infinite.
First we consider the special case where the base complex $L$ is a simplex.
\begin{lemma}\label{l:simplex}
	Suppose $V_J$ is the $J$-fold product of $\{V_s\}_{s\in J}$, where each $V_s$ is infinite.
	For each $I\in\cs(V_J,T(J))$, $\cl^I(J)$ denotes the $J$-fold join defined by \eqref{e:li}.
	Then
	\[
		H^n(V_J;\zz V_J)=\bigoplus_{I\in \cs(V_J,T(J))} \rh^{n-1} (\cl^I(J))\otimes \hat{A}(V_J)^I.
	\]
\end{lemma}
\begin{proof}
	Let $\ck(J)$ denote the Davis chamber for $(V_J,T(J))$.
	By Theorem~\ref{t:d98},
	\[
		H^n(V_J;\zz V_J)=\bigoplus_{I\in \cs(V_J,T(J))} H^n(\ck(J),\ck^{T(J)-I}(J)) \otimes \hat{A}(V_J)^I.
	\]
	Moreover, $\ck^{T(J)-I}(J)$ is homotopy equivalent to $\cl^I(J)$.
	The formula in the lemma follows.
\end{proof}

Finally, consider the general case when each $V_s$ is infinite:
\begin{align}
	\grh^n(V;\zz V)&= \bigoplus_{I\in \cs(V,T)} H^n(\ck, \ck^{T-I})\otimes \hat{A}(V)^I \notag\\
	&=\bigoplus_{I\in \cs(V,T)}\ \bigoplus_{\substack{J\in \cs(W,S)\\i+j=n}}H^i(K_J,\partial K_J;\rh^{j-1}\bigl(\cl^I(J))\bigr) \otimes \hat{A}(V)^I \notag\\
	&=\bigoplus_{\substack{J\in \cs(W,S)\\i+j=n}}H^i(K_J,\partial K_J; \bigoplus_{I\in \cs(V,T)} \rh^{j-1}(\cl^I(J)) \otimes \hat{A}(V)^I) \notag\\
	&=\bigoplus_{\substack{J\in \cs(W,S)\\i+j=n}}H^i(K_J,\partial K_J;H^j(V;\zz V)),\label{e:infinite}
\end{align}
where the first equation follows from Theorem~\ref{t:d98}, the second from Corollary~\ref{c:pjoin2}, the third from the fact that $\hat{A}(V)^I$ is free abelian, and the last from Lemma~\ref{l:simplex}.
Moreover, \eqref{e:infinite} agrees with Theorem~\ref{t:gpfree}.

\section{Weighted \texorpdfstring{$L^2$}{L\texttwosuperior}-cohomology of buildings and Coxeter groups}\label{s:weighted}
\subsection{Hecke--von Neumann algebras}\label{ss:hvn}
Suppose given a Coxeter system $(W,S)$ and a function $i:S\to I$ to an index set $I$ such that $i(s)=i(t)$ whenever $s$ and $t$ are conjugate in $W$.
A \emph{multiparameter} for $(W,S)$ is an $I$-tuple $\bt=(t_i)_{i\in I}$ of indeterminates (or of numbers).
Write $t_s$ instead of $t_{i(s)}$.
If $s_1\cdots s_n$ is a reduced expression for an element $w\in W$, then the monomial
\[
	\bt_w:= t_{s_1}\cdots t_{s_n}
\]
depends only on $w$ and not on the choice of reduced expression for it.
(This follows from Tits' solution to the Word Problem for Coxeter groups, cf.~\cite[p.~315]{dbook}.) The \emph{growth series} of $W$ is power series in $\bt$ defined by,
\[
	W(\bt):=\sum_{w\in W} \bt_w.
\]
This power series has a region of convergence $\car(W)$ (a subset of $\bC^I$).
If $W$ is finite, then $W(\bt)$ is a polynomial.
For any Coxeter group $W$, it can be shown that $W(\bt)$ is a rational function of $\bt$ (cf.\ \cite[Cor.~17.1.6]{dbook}).

For any set $X$, $\bR^X$ denotes the vector space of finitely supported real-valued functions on $X$.
For each $x\in X$, $e_x$ denotes the indicator function of $\{x\}$ so that $\{e_x\}_{x\in X}$ is the \emph{standard basis} for $\bR^X$.
For a multiparameter $\bq$ of positive real numbers, define a inner product $\langle\ ,\ \rangle_\bq$ on $\bR^W$ by
\begin{equation*}
	\langle e_w,e_{w'}\rangle_\bq=
	\begin{cases}
		\bq_w &\text{if $w=w'$,}\\
		0 &\text{otherwise.}
	\end{cases}
\end{equation*}
Let $\ltwo (W)$ denote the Hilbert space completion of $\bR^W$ with respect to $\langle\ ,\ \rangle_\bq$.

Using $\bq$, one can also give $\bR^W$ the structure of a \emph{Hecke algebra}, determined by the formula
\[
	e_s e_w=
	\begin{cases}
		e_{sw} &\text{if $l(sw)>l(w)$,}\\
		q_s e_{sw} +(q_s-1)e_w &\text{if $l(sw)<l(w)$.}
	\end{cases}
\]
When $\bq=\bone$ (the multiparameter which is identically $1$), $\bR_\bq(W)$ is the group algebra of $W$.

Define an anti-involution $*$ on $\bR_\bq W$, by $(\sum x_we_w)^* := \sum x_{w^\minus}e_w$.
The inner product $\langle\ ,\ \rangle_\bq$ and the anti-involution $*$ give $\bR_\bq W$ the structure of a Hilbert algebra (see \cite[Prop.~2.1]{dym}).
(In other words, $x^*$, the image of $x$ under the anti-involution of algebras, is equal to the adjoint of $x$ with respect to $\langle\ ,\ \rangle_\bq$.) This implies that there is an associated von Neumann algebra $\cn (W)$ (called the \emph{Hecke--von Neumann algebra}) acting from the right on $\ltwo (W)$.
One definition of $\cn (W)$ is that it is the algebra of all bounded linear endomorphisms of $\ltwo (W)$ which commute with the left $\bR_\bq (W)$-action.
An equivalent definition is that it is the weak closure of the elements $\bR_\bq(W)$ which act from the right on $\ltwo (W)$ as bounded linear operators.

Define the \emph{von Neumann trace} of $\gf \in \cn (W)$ by $\tr_{\cn}(\gf): =\langle e_1\gf, e_1\rangle_\bq$ and similarly, for any $(n\times n)$-matrix with coefficients in $\cn(W)$.
This allows us to define the \emph{von Neumann dimension} of any closed subspace of an $n$-fold orthogonal direct sum of copies of $\ltwo (W)$ which is stable under the diagonal $\bR_\bq (W)$-action: if $V\subset (\ltwo(W))^n$ is such a subspace and $p_V: (\ltwo(W))^n\to (\ltwo(W))^n$ is orthogonal projection onto $V$, then $p_v\in \cn (W)$, so define
\begin{equation}\label{e:dim}
	\dim_{\cn} V:=\tr_{\cn} (p_V).
\end{equation}

For any $J\subset S$ and $\bq\in \car (W_J)$, there is a self-adjoint idempotent $a_J\in \cn (W)$ defined by
\begin{equation*}
	a_J:=\frac{1}{W_J(\bq)} \sum_{w\in W_J} e_w
\end{equation*}
(cf.~\cite[Lemma 19.2.5]{dbook}).
For $s\in S$, write $a_s$ instead of $a_{\{s\}}$.
For $s\in S$ and $J\subset S$, define subspaces of $\ltwo (W)$ by
\[
	A^s:=\ltwo (W)a_s\ ,\quad\quad A^J:=\bigcap_{s\in J}A^s.
\]
These subspaces are stable under the action of $\bR_\bq W$ from the left.
Moreover, $A^J$ is the image of the idempotent $a_J$ if $\bq\in \car(W_J)$, and $A^J=0$ whenever $\bq\notin \car(W_J)$ (cf.~\cite[\S 19.2]{dbook}).

Let $A^{>J}$ denote the subspace $\sum_{I>J} A^I$ of $A^J$ and put
\begin{equation*}\label{e:dj}
	D^J:= A^J \cap (A^{> J})^\perp
\end{equation*}
The following is one of the main results of \cite{ddjo07} (or see \cite[Thm~20.6.1]{dbook}).
\begin{theorem}[The Decomposition Theorem of \cite{ddjo07}]\label{t:ddjodecomp}
	If $\bq\in \ol{\car} \cup \ol{\car^\minus}$, then
	\[
		\sum_{I\supseteq J} D^I
	\]
	is a direct sum and a dense subspace of $A^J$.
	In particular, taking $J=\emptyset$,
	\[
		\ltwo= \ol{\sum D^I}.
	\]
	If $\bq\in \ol{\car}$, the only nonzero terms in this sum are those with $I$ cospherical (i.e., with $S-I\in \cs$), and if $\bq^\minus\in \ol{\car}$, the only nonzero terms are those with $I$ spherical.
	Moreover, for $\bq^\minus \in \ol{\car}$,
	\[
		\dim_{\cn} D^J=\sum_{I\in \cs_{\ge J}} \frac{(-1)^{|I-J|}}{W_I(\bq)}\ .
	\]
\end{theorem}

\subsection{Weighted \texorpdfstring{$L^2$}{L\texttwosuperior}-Betti numbers}\label{ss:betti}
Suppose $M$ is a mirrored CW complex over $S$.
Let $\bp$ be a multiparameter of positive real numbers for a Coxeter system $(W,S)$.
Define a measure $\mup$ on the set of cells of $\cu(W,M)$ by $\mup (c):=\bp_w$, where $w\in W$ is the element of shortest length which moves $c$ into the base chamber $M$ ($=$ the image of $1\times M$ in $\cu(W,M)$).
As in \cite{ddjo07} or \cite{dbook}, we can use $\mup$ to define the \emph{weighted $L^2$-cochains}, $\ltwop C^*(\cu(W,M))$.
The corresponding reduced cohomology groups are denoted $\ltwop H^*(\cu (W,M))$.
The weighted cochains on $\cu(W,M)$ can also be regarded as cochains on $M$ with respect to a certain coefficient system $\ci(\ltwop)$.
This coefficient system associates to a cell $c$ in $M$, the left $\cnp$-module, $\ltwop (W)a_{S(c)}$, i.e.,
\[
	\ci(\ltwop)(c):= \ltwop (W)a_{S(c)}.
\]
The corresponding cochain complex is denoted $C^*(M;\ci(\ltwop))$ and its reduced cohomology by $H^*(M;\ci(\ltwop))$.
We have natural identifications,
\begin{align*}
	C^*(M;\ci(\ltwop))&\cong \ltwop C^*(\cu(W,M)),\\
	H^*(M;\ci(\ltwop))&\cong \ltwop H^*(\cu(W,M)).
\end{align*}
(This is completely analogous to \eqref{e:coeff} and \eqref{e:coeff2} of \S\ref{s:cohom}.
See \cite{ddjo2, ddjo07}.) The \emph{$j^{th}$-weighted $L^2$-Betti number} of $\cu(W,M)$ is defined by
\[
	\ltwop b^j(\cu(W,M)):=\dim_{\cnp} \ltwop H^j(\cu (W,M)) = \dim_{\cnp}H^j(M; \ci(\ltwop)),
\]
where $\dim_{\cnp}$ is defined by \eqref{e:dim}.
Also, put
\[
	\ltwop b^j(W):=\ltwop b^j(\cu(W,K)).
\]
We recall some of the main results of \cite{ddjo07} and \cite{dym}.
\begin{theorem}[Dymara \cite{dym}]\label{t:dym}
	Suppose $\bp\in \ol{\car}$.
	Then $\ltwop b^j(W)=0$ for $j>0$, while
	\[
		b^0(W)=\dim_{\cN_\bp} A^S= \frac{1}{W(\bp)}.
	\]
\end{theorem}
\begin{theorem}[{\cite[Thm.~10.3]{ddjo07}}]\label{t:ddjo}
	Suppose $\bp^\minus \in \ol{\car}$.
	Then
	\[
		\ltwop b^j(\cu(W,M))=\sum_{J\in \cs(W,S)} b^j(M,M^{S-J}) \dim_{\cnp} D^J,
	\]
	where the formula for $\dim_{\cnp} D^J$ is given in Theorem~\ref{t:ddjodecomp}.
\end{theorem}
\begin{theorem}[{\cite[Thm.~13.8]{ddjo07}}]\label{t:ddjobldg}
	Suppose $(\cac,\gd)$ is a locally finite building of type $(W,S)$ with a chamber transitive automorphism group $G$ and that the thickness of $\cac$ is given by a multiparameter $\bp$ of integers.
	(In other words, each $s$-panel contains $p_s+1$ chambers.) Let $M$ be a mirrored CW complex over $S$ (with a $W$-finite mirror structure).
	Then
	\[
		L^2b^j(\cb(\cac,M),G)=\ltwop b^j(\cu(W,M)).
	\]
\end{theorem}
\noindent (The group $G$ need not be discrete.
The von Neumann dimensions with respect to $G$ are defined using Haar measure on $G$, normalized so that the stabilizer of a chamber has measure 1.)
\begin{corollary}\label{cor:ddjofinite}
	Let $G$ be a graph product of finite groups $\{G_s\}_{s\in S}$.
	Let $\bp$ be the multiparameter defined by $p_s:=|G_s|-1$.
	Then
	\[
		L^2b^j(G)=\sum _{J\in \cs(W,S)} \ltwop b^j(K,K^{S-J}).
	\]
\end{corollary}

As in the last paragraph of \S\ref{s:cohom}, there is a different method which can be used to define weighted $L^2$-Betti numbers by using ideas of L\"uck \cite{luck}.
As in \cite{luck} there is an equivalence of categories between the category of Hilbert $\cnp (W)$-modules (i.e.
$\cnp(W)$-stable closed subspaces of $\ltwop(W)^n$) and the category of ordinary projective modules for $\cnp(W)$.
This allows us to define a ``dimension,'' $\dim_{\cnp} M$, of a finitely generated, projective $\cnp(W)$-module which agrees with the dimension of the corresponding Hilbert $\cnp(W)$-mod\-ule.
The $\cnp(W)$-dimension of an arbitrary $\cnp(W)$-module is then defined to be the dimension of its projective part.

There is a coefficient system $\ci(\cnp)$ on the mirrored CW complex $M$ defined by
\[
	\ci(\cnp)(c):= \cnp(W)a_{S(c)}.
\]
The corresponding cohomology groups are denoted $H^*(M;\ci(\cnp))$.
The dimension of $H^j(M;\ci(\cnp))$ is equal to that of $H^j(M;\ci(\ltwop))$ (and they are both equal to $j^{th}$-weighted $\ltwop$-Betti number of $\cu(W,M)$).
(The advantage of using the coefficient system $\ci(\cnp)$ instead of $\ci(\ltwop)$ is that it is not necessary to use reduced cohomology and then have to keep taking closures of images.)

\section{Weighted \texorpdfstring{$L^2$}{L\texttwosuperior}-Betti numbers of graph products of Coxeter groups}\label{s:lgp}

As in \S\ref{ss:oct}, $(W_L,S)$ is the RACS associated to a graph $\gG$, $\{(V_s,T_s)\}_{s\in S}$ is a family of Coxeter systems and $(V,T)$ is the corresponding graph product of Coxeter systems.
Let $\bq$ be a multiparameter for $(V,T)$.
It restricts to a multiparameter for each $V_s$, which we will denote by the same letter.
By Lemma~\ref{l:RAB}, $V$ is a RAB of type $(W_L,S)$.

Let $\bp$ be the multiparameter for $(W_L,S)$ given by $p_s = V_s(\bq) -1$.
The following lemma shows that the growth series of $(V,T)$ and $(W_L,S)$ are related by a change of variables $\bq\to \bp$.
\begin{Lemma}\label{l:growth}
	For $w\in W_L$,
	\begin{equation}\label{l:growth1}
		\sum_{v\in \pi^{-1}(w)} \bq_v=\bp_w,
	\end{equation}
	and, therefore,
	\[
		V(\bq)=W(\bp).
	\]
\end{Lemma}
\begin{proof}
	Let $s_1\cdots s_n$ be a reduced expression for $w\in W$ and let $v\in \pi^{-1}(w)$.
	Then $v$ factors as a product $v_{s_1}\cdots v_{s_n}$, with $v_{s_i} \in V_{s_i}^*$, and this factorization gives one-to-one correspondence between $\pi^{-1}(w)$ and $V_{s_1}^* \times \cdots \times V_{s_n}^*$; moreover, $\bq_v=\bq_{v_{s_1}}\cdots \bq_{v_{s_n}}$.
	(Recall from \S\ref{ss:spaces} that $V^*_s=V_s-\{1\}$.) Hence, the growth series of $\pi^{-1}(w)$ is the product of the growth series of the $V_{s_i}^*$, and the result follows.
\end{proof}

In the first subsection we compute the weighted $L^2$-Betti numbers of $V$ in the case where $\bq\notin \car(V_s)$ for each $s\in S$.
Notice that this necessarily entails that each $V_s$ is infinite.
The proof uses the spectral sequence of \S\ref{s:ss} in the same way as in \S\ref{ss:cgfree}.
In the second subsection we consider the opposite situation where $\bq\in \ol{\car(V_s)}$ for each $s\in S$.
For example, this holds for all $\bq$ when each $V_s$ is finite.
In this case the proofs are based on arguments from \cite{ddjo07}.

\subsection{Large weights}\label{largeweight}
In this subsection \emph{we assume $\bq\notin \car(V_s)$ for each $s\in S$}.
\begin{theorem}\label{t:largeq}
	\[
		\ltwo b^n(V)=\sum_{\substack{i+j=n\\J\in \cs(W,S)}} b^i(K_J,\partial K_J)\cdot \ltwo b^j(V_J),
	\]
	where
	\[
		\ltwo b^j(V_J)=\prod _{\substack{\sum k(s)=k\\s\in J}} \ltwo b^{k(s)}(V_s).
	\]
\end{theorem}
\begin{proof}[Proof of Theorem~\ref{t:largeq}] The proof is almost the same as the proof of Theorem~\ref{t:gpfree}.
	Since $(V,T)$ is a Coxeter system, we prefer to use its natural action on its Davis complex rather than on $EV$.
	Let $Y:=\cu(V,\ck)$ be the Davis complex and for each $J\in \cs(W,S)$, put
	\[
		Y'_J:=\cu(V_J,\ck(J)), \quad Y_J:=V\times _{V_J} Y'_J.
	\]
	As before,
	\[
		Y'_J=\prod_{s\in J}Y'_s,\quad \text{where}\quad Y'_s:=\cu(V_s,\ck(s)).
	\]
	Then $\cv=\{Y_J\}_{J\in \cs(W,S)}$ is a poset of spaces on $Y$.
	The spectral sequence of \S\ref{s:ss} has $E_1^{i,j}=C^j(K;\ch^j(\cv))$, where the coefficient system is defined by $\gs\mapsto H^j_V(Y_{\min \gs};\cn (V))$.
	It converges to $H^*_V(Y;\cn (V))$ and the $\cn$-dimensions of these cohomology spaces are the $\ltwo$-Betti numbers.
	Since $\bq\notin \car(V_s)$ for each $s\in S$, $H^0_{V_s}(Y'_s;\cn (V_s))=0$ by \cite{dym}.

	and the relative K\"unneth Formula gives that
	\[
		H^*_{V_J}(Y'_J;\cn (V_J))\to H^*_{V_J}(Y'_{<J};\cn (V_J))
	\]
	is the zero map.
	By Lemma~\ref{l:main},
	\[
		E_2^{i,j}=\bigoplus_{J\in \cs(W,S)}H^i(K_J,\partial K_J)\otimes \cn (V)
	\]
	and the spectral sequence degenerates at $E_2$.
	Taking von Neumann dimensions, we get the formula for weighted $L^2$-Betti numbers.
	The last formula also follows from the K\"unneth Formula.
\end{proof}

We also have a weighted version of Theorem~\ref{t:other}.
Let $V'$ denote the direct sum $\prod_{s\in S}V_s$.
A Coxeter system $(V'',T)$ \emph{lies between $V$ and $V'$} if its presentation is given by changing certain entries of its Coxeter matrix from $\infty$ to an even integers $\ge 2$ , more specifically, for any $(t_1, t_2)$ with $t_i\in T_{s_i}$ and $\{s_1,s_2\}\notin \edge (\gG)$, we are allowed change $m(t_1, t_2)$ from $\infty$ to an even integer.
Note that a multiparameter $\bq$ for $(V,T)$ is also a multiparameter for $(V',T)$ and for $(V'', T)$.
The proof of Theorem~\ref{t:largeq} also gives the following.
\begin{theorem}[cf.~Theorem~\ref{t:other}]\label{t:otherw}
	\[
		\ltwo b^n(\cu(V'', \ck))=\sum_{\substack{i+j=n\\J\in \cs}} b^i(K_J,\partial K_J)\cdot \ltwo b^j(V_J).
	\]
\end{theorem}

\subsection{Small Weights}\label{smallweight}

Throughout this subsection we suppose that the multiparameter $\bq$ is ``small'' in the sense that $\bq \in \car(V_s)$ for each $s\in S$.
Let $M$ be a mirrored CW complex over $S$ and let $\cb(V,M)$ be the $M$-realization of $V$ defined by \eqref{e:real} of \S\ref{ss:real}.

As before, define a measure $\muq$ on the set of cells of $\cb(V,M)$ by putting $\muq(c):=\bq_v$, where $v\in V$ is the shortest element such that $vc$ lies in the base chamber $M$.
Again, we get a cochain complex, $\ltwo C^*(\cb(V,M))$.
(If $Y$ is a mirrored CW complex over $T$, then associated to the Coxeter system $(V,T)$ there is a different cochain complex, $\ltwo C^*(\cu(V,Y))$.) There is a coefficient system $\ci(\cn)$ on $M$ defined by
\[
	\ci(\cn)(c):= \cn(V)a_{\pi^\minus (S(c))}
\]
and the $\cn(V)$-dimension of $H^j(M;\ci(\cn))$ is $\ltwo b^j (\cb(V,M))$.

Let $K$ and $\ck$ denote the geometric realizations of $\cs(W,S)$ and $\cs(V,T)$, respectively.
\begin{theorem}\label{t:UV}
	\qquad \( \ltwo b^*(\cu(V,\ck))=\ltwo b^*(\cb(V,K)).
	\)
\end{theorem}
\begin{proof}
	We use the same spectral sequence as in the proof of Theorem~\ref{t:largeq}.
	It converges to $H^*_V(\cu(V,\ck);\cn (V))$ and has $E_1$-term:
	\[
		E^{i,j}_1= C^j(K;\ch^j(\cv)).
	\]
	where the coefficient system is defined by $\gs\mapsto H^j_V(X_{\min \gs};\cn (V))$.
	Since $\bq\in \car(V_s)$, for each $s\in S$, by Dymara's result, Theorem~\ref{t:dym}, the coefficients are nonzero only for $j=0$.
	For $j=0$ the coefficient system is associated to the poset of coefficients $J\mapsto \cn a_{\pi^\minus(J)}$.
	In \S\ref{ss:betti} we denoted this coefficient system by $\ci(\cn)$.
	So, $E^{i,0}_1$ is the cochain complex $C^j(K;\ci(\cn))$, in other words, it is the cochain complex whose cohomology gives the $b^*_\bq(\cb(V,K))$.
	Thus,
	\[
		b^*_\bq(\cu(V,\ck))=b^*_\bq(\cb(V,K)).
		\qedhere
	\]
\end{proof}
\begin{Remark}
	Suppose each $V_s$ is finite.
	Then each link $\cl(s)$ is a simplex, and it follows that the natural map $\cl\to L$, induced by $\pi$, has contractible fibers.
	It follows that $\ck$ deformation retracts to $K$, respecting the mirror structure.
	This deformation retraction induces a $V$-equivariant, proper homotopy equivalence $\cu(V,\ck)\to \cb(V,K)$.
	So, when each $V_s$ is finite, Theorem~\ref{t:UV} is the expected result.
\end{Remark}
\begin{theorem}\label{t:bldg}
	\( \ltwo b^*(\cb(V,M))=\ltwop b^*(\cu(W,M)).
	\)
\end{theorem}
\begin{Remark}
	If each $V_s$ is finite and $\bq=\bone$, then $V$ is a locally finite building of type $(W,S)$ of thickness $\bp_s=|V_s|-1$; so this reduces to Theorem~\ref{t:ddjobldg} (i.e., \cite[Thm.~13.8]{ddjo07} or \cite[Thm.~20.8.4]{dbook}).
	The key point of the proof, which goes back to \cite{dym}, is that the folding map $\pi$ pulls back $p$-weighted harmonic cochains to $q$-weighted harmonic cochains.
	The proof of Theorem~\ref{t:bldg} is a minor generalization of the proof of \cite[Thm.~13.8]{ddjo07} to locally infinite buildings.
	It occupies the end of this subsection.
\end{Remark}
\begin{example}
	Figure~\ref{f:bldg} depicts the folding map $\pi:\cb(V,M) \to \cu(W,M)$ in the case of the free product of two infinite dihedral groups.
	Here the graph $\gG$ is two disjoint points $s$ and $t$, so $W$ is $\bD_{\infty}$ generated by $s$ and $t$.
	The vertex groups are also $\bD_{\infty}$, generated by $\{s^+, s^-\}$ and $\{t^+, t^-\}$, respectively.
	So, $V=\bD_{\infty}*\bD_{\infty}$, and we let $M$ be a segment $K$.
\end{example}
\begin{figure}[ht]
	\newcommand{\lvld}{\textwidth/6}
	\begin{tikzpicture}[scale=1,level distance=\lvld, level 1/.style={sibling distance=24mm}, level 2/.style={sibling distance=5mm}, grow=left]

		\draw[dotted,thick] (-\lvld/2,30mm) --++(0,15pt); \draw[dotted,thick] (-\lvld/2,-30mm) --++(0,-15pt); \draw[dotted,thick] (3*\lvld/2,30mm) --++(0,15pt); \draw[dotted,thick] (3*\lvld/2,-30mm) --++(0,-15pt); \coordinate child foreach \x in {t^+t^-,t^+,t^-,t^-t^+} {child foreach \y in {s^+s^-,s^+,s^-,s^-s^+} { node { $\bq^{\x\y }$}} edge from parent node[below,pos=0.6] {$\bq^{\x}$ }}; \coordinate child[level 2/.style={sibling distance=24mm}, level 3/.style={sibling distance=5mm}, grow=right] {child foreach \x in {s^+s^-,s^+,s^-,s^-s^+} { child foreach \y in {t^+t^-,t^+,t^-,t^-t^+} { node { $\bq^{\x\y }$}} edge from parent node[below,pos=0.6] {$\bq^{\x}$ }} edge from parent node[above] {$\bq^e$}};

		\foreach \x in {-2,...,3} \draw[xshift=\lvld *\x , yshift=-53mm] (0pt,-2pt) -- (0pt,2pt); \foreach \x/\y in {-2/ts,-1/t,0/e,1/s,2/st} \draw[xshift=\lvld *\x , yshift=-53mm] (0,0) -- node[below]{$\bp^{\y}$}(\lvld ,0pt);

		\draw[->] (\lvld/2, -10mm)--node[auto=left] {$\pi$} (\lvld/2, -43mm);
	\end{tikzpicture}
	\caption{$\bD_{\infty}*\bD_{\infty}$ as a locally infinite building over $\bD_{\infty}$.\label{f:bldg}
	}
\end{figure}
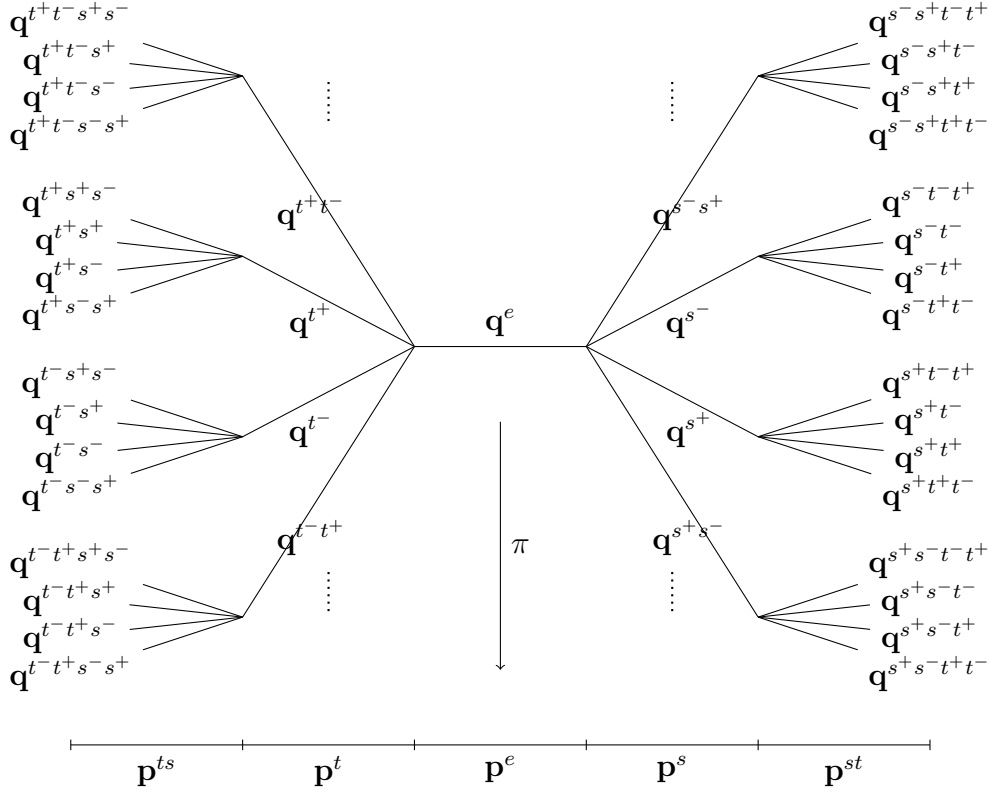

Combining the two previous theorems we get the following.
\begin{theorem}\label{c:main}
	\quad \( \ltwo b^*(V)=\ltwop b^*(W).
	\)
\end{theorem}
\begin{proof}
	\begin{align*}
		\ltwo b^*(V) &:=\ltwo b^*(\cu(V,\ck))= \ltwo b^*(\cb(V,K)) &\text{(by Theorem~\ref{t:UV})}\\
		&\ = \ltwop b^*(\cu(W,K)) &\text{(by Theorem~\ref{t:bldg})}\\
		&:=\ltwop b^*(W).
		& & \qedhere
	\end{align*}
\end{proof}
\begin{remark}\label{r:gpbldg}
	(\emph{Graph products of spherical buildings}).
	Suppose $\{\cac_s\}_{s\in S}$ is a family of buildings where $\cac_s$ is type $(V_s,T_s)$.
	In \cite[Ex.~3.1~(3)]{d09} the first author defined the notion of a ``graph product of buildings,'' $\prod_{\gG} \cac_s$.
	It is a building of type $(V,T)$.
	Suppose each $\cac_s$ is spherical of thickness $\bq_s$ with a chamber-transitive automorphism group $G_s:=\Aut (\cac_s)$ (i.e., each $V_s$ is a finite Coxeter group and the number of chambers of $\cac_s$ in a panel of type $t_s$ is $q_{t_s}+1$).
	By Theorem~\ref{t:ddjobldg}, the ordinary $L^2$-Betti numbers of $\cac:=\prod_{\gG} \cac_s$ with respect to $G:=\prod_{\gG} G_s$ are given by
	\[
		L^2b^j(\cb(\cac, \ck))=\ltwo b^j(V)=\ltwop b^j(W_L),
	\]
	where $p_s=V_s(\bq)-1=|\cac_s|-1$.
	In other words, the $L^2$-Betti numbers of a graph product of spherical buildings depend only on the thickness of the buildings and the weighted $L^2$-Betti numbers of the associated RACS, $(W,S)$.
\end{remark}

\paragraph{The proof of Theorem~\ref{t:bldg}} The proof is a modification of the proof in \cite[Theorem 13.8 in Section 13]{ddjo07} and follows a series of lemmas.
\begin{lemma}\label{l:vnsubalgebra}
	\begin{enumeratei}
		\item\label{l:vnsubalgebrai}
		The map $\pi:V\to W$ induces an isometric embedding $\pi^*:\ltwop(W) \to \ltwo (V)$.
		\item\label{l:vnsubalgebraii}
		For each $s\in S$, $\pi^*(a_s)=a_{T_s}$.
		Moreover, for each spherical subset $J\subset S$, $\pi^*(a_J)=a_{\pi^\minus (J)}$.
		\item\label{l:vnsubalgebraiii}
		The map $\pi^*:\ltwop (W)\to \ltwo (V)$ induces a monomorphism of von Neumann algebras $\pi^*:\cN_\bp(W) \to \cn(V)$.
		(In particular, $\pi^*$ commutes with the $*$ anti-involutions on $\cN_\bp(W)$ and $\cn(V)$.)
	\end{enumeratei}
\end{lemma}
\begin{proof}
	To prove (\ref{l:vnsubalgebrai}), notice that as $w$ varies over $W$, the vectors $\pi^*(e_w)$ are orthogonal to each other, and equation (\ref{l:growth1}) implies that $\norm{\pi^*(e_w)}_\bq=\norm {e_w}_\bp$.

	Statement (\ref{l:vnsubalgebraii}) follows immediately from the definitions.

	The idempotents $a_s$ and $a_r$, with $r,s\in S$, commute if and only if $r$ and $s$ commute.
	So, if $a_s$ commutes with $a_r$, then $a_{T_s}$ commutes with $a_{V_r}$.
	Since the $a_s$ generate the Hecke algebra, statement~(\ref{l:vnsubalgebraiii}) follows from (\ref{l:vnsubalgebrai}) and (\ref{l:vnsubalgebraii}).
\end{proof}

Similarly to the equation~(\ref{l:growth1}), the measures $\muq$ and $\mup$ on the cells of $\cb(V,M)$ and $\cu(W,M)$ are related by
\begin{equation}\label{eq:mumu}
	\sum_{c'\in \pi^{-1}(c)} {\muq(c')}=\mup (c).
\end{equation}

By Lemma \ref{l:vnsubalgebra}, the map $\pi:\cb(V,M)\to \cu(W,M)$ induces a cochain map $\pi^*:\ltwop C^\ast(\cu(W,M)) \to \ltwo C^\ast(\cb(V,M))$.
We also have a ``transfer map'' $t:\ltwo C^\ast(\cb(V,M))\to \ltwop C^\ast(\cu(W,M))$ defined by
\[
	t(f)(c):=\sum_{c'\in \pi^{-1}(c)} f(c')\frac{\muq(c')}{\mup(c)}.
\]
\begin{lemma}\label{l:randt}
	\begin{enumeratei}
		\item $t\circ \pi^*=id:\ltwop C^i(\cu(W,M))\to\ltwop C^i(\cu(W,M))$.
		\item The maps $\pi^*$ and $t$ are adjoint to each other.
		\item These maps take harmonic cocycles to harmonic cocycles.
	\end{enumeratei}
\end{lemma}
\begin{proof}
	Statement (i) is obvious.

	(ii) For $f\in \ltwop C^i(\cu(W,M))$ and $f'\in \ltwo C^i(\cb(V,M))$, we have
	\begin{align*}
		\langle \pi^*(f),f'\rangle_\bq &= \sum_{c'\in \cb^{(i)}}[\pi^*(f)(c')][f'(c')]\muq(c')\\
		&=\sum_{c'\in \cb^{(i)}} f(\pi(c'))f'(c')\muq(c')\\
		&=\sum_{c\in \cu^{(i)}}f(c)\sum_{c'\in \pi^{-1}(c)}f'(c')\muq(c')\\
		&=\sum_{c\in\cu^{(i)}}[f(c)][t(f')(c)]\mup(c) \\
		&=\langle f,t(f')\rangle_\bp,
	\end{align*}
	where $\cb^{(i)}$ and $\cu^{(i)}$ denote the set of $i$-cells in $\cb (V,M)$ and $\cu(W,M)$, respectively.

	(iii) Since $\pi^*:\ltwop C^\ast(\cu(W,M))\to \ltwo C^\ast(\cb(V,M))$ is induced by the cellular map $\pi:\cb(V,M)\to \cu(W,M)$, it takes cocycles to cocycles.
	We must show it also takes cycles to cycles.
	If $c'\in \cb^{(i-1)}$ and $d'\in \cb^{(i)}$ and if the incidence number $[c':d']$ is nonzero, then it is equal to $[\pi(c'):\pi(d')]$.
	Hence,
	\begin{equation*}
		\partial^\bq (\pi^*(f))(c')=\sum_{d'}[c':d']\frac{\muq (d')}{\muq (c')}f(\pi(c'))=\sum_{d}[c:d]\frac{\mup (d)}{\mup (c)}f(c)=\partial^\bp (f)(c),
	\end{equation*}
	where $c=\pi(c')$, $d=\pi(d')$, the first and the last equality come from the definition, and the middle equality comes from equation \eqref{eq:mumu}.
	So, $\partial^\bp (f)=0$ implies that $\partial^\bq (\pi^*(f))=0$.
	Since $t$ is the adjoint of $\pi^*$, it also must take cocycles to cocycles and cycles to cycles.
\end{proof}

Consider the diagram:
\[
	\minCDarrowwidth 0.5cm
	\begin{CD}
		\bigoplus \ltwop (W)@>\oplus a_{S(c)}>> \bigoplus A_{S(c)}= \ltwop C^\ast(\cu(W,M)) @>P>> \ltwop \ch^\ast(\cu(W,M))\\
		@V\pi^* VV @V\pi^* VV @V\pi^* VV\\
		\bigoplus \ltwo (V)@>\oplus a_{\pi^{-1}{(S(c))}}>> \bigoplus A_{\pi^{-1}(S(c))}=\ltwo C^\ast(\cb(V,M)) @>P>> \ltwo \ch^\ast(\cb(V,M))
	\end{CD}
\]
where $P$ denotes the orthogonal projection onto harmonic cocycles.
\begin{lemma}\label{l:rp}
	The above diagram commutes.
\end{lemma}
\begin{proof}
	The commutativity of the first square follows from Lemma~\ref{l:vnsubalgebra}.

	Let $x\in \ltwop C^\ast(\cu(W,M))$.
	To prove commutativity of the second square, it is enough to show that $P\pi^*(x)-\pi^* P(x)$ is orthogonal to any harmonic cocycle $h\in \ltwo \ch^\ast (\cb(V,M))$.
	We have: $\langle P\pi^*(x),h\rangle_\bq =\langle \pi^*(x),P(h)\rangle_\bq= \langle \pi^*(x),h\rangle_\bq$.
	Hence,
	\[
		\langle P\pi^*(x)-\pi^* P(x),h\rangle_\bq=\langle \pi^*(x-P(x)),h\rangle_\bq= \langle x-P(x),t(h)\rangle_\bp=0,
	\]
	where the second and third equalities follow, respectively, from parts (ii) and (iii) of Lemma~\ref{l:randt}.
\end{proof}
\begin{proof}[Proof of Theorem~\ref{t:bldg}] Let $e_c\in\bigoplus \ltwo (V)$ denote the unit vector $e_1 \in \ltwo (V)$ in the summand corresponding to a cell $c \in M^{(i)}$, and similarly for $\ltwop (W)$.
	Note that $\pi^*(e_c)=e_c$.
	Let $\theta$ denote the compositions of the maps in the top and the bottom rows of the above diagrams, i.e., $\theta$ is the orthogonal projection of the free Hecke--von Neumann module onto harmonic cocycles.
	Using Lemmas~\ref{l:rp} and \ref{l:randt}, we get
	\begin{align*}
		b^i_\bq(\cb(V,M))&:=\dim_{\cn(V)}\ltwo H^i(\cb(V,M))\\
		&=\sum \langle \theta(e_c),e_c\rangle_\bq =\sum \langle \theta \pi^*(e_c),\pi^* (e_c)\rangle_\bq \\
		& =\sum \langle \pi^* \theta(e_c),\pi^*(e_c)\rangle_\bq = \sum \langle \theta(e_c),t\pi^*(e_c)\rangle_\bp\\
		&= \sum \langle \theta(e_c),(e_c)\rangle_\bp =\dim_{\cN_p(W)}\ltwop H^i(\cu(W,M))\\
		&:=b^i_\bp(\cu(W,M)).
		\qedhere
	\end{align*}
\end{proof}

\section{Octahedralization}\label{s:oct}
Suppose $L$ is a simplicial complex.
Its \emph{octahedralization}, $\OL$, is defined by
\begin{equation}\label{e:defoct}
	\OL:=\aster_L S^0.
\end{equation}

Next we work out an example which motivated most of the the calculations in this paper.
For each $s\in S$, $V_s$ is the infinite dihedral group with generating set, $T_s:=\{s^+,s^-\}$.
Suppose $(W_L,S)$ is the RACS associated to the graph $\gG$ and $(V,T)$ is the graph product of the infinite dihedral groups (so that $(V,T)$ is also a RACS. By \eqref{e:lvt}, $L(V,T)=OL$.
So, in this special case we shall write $W_\OL$ for $V$ and $OS$ for $S$ and call the RACS, $(W_\OL,OS)$, the \emph{octahedralization} of $(W,S)$.

Theorem~\ref{t:gpfree} gives the following calculation of the cohomology of $W_\OL$ with group ring coefficients.
\begin{theorem}\label{c:freeoct}
	\begin{equation*}\label{e:wol}
		\grh^n(W_\OL;\zz W_\OL)=\bigoplus_{J\in \cs(W,S)} H^{n-|J|}(K_J,\partial K_J) \otimes \zz (W_\OL/W_{OJ}).
	\end{equation*}
\end{theorem}
\begin{proof}
	Since $\bD_{\infty}$ acts properly and cocompactly on $\bR^1$, it is a $1$-dimensional virtual Poincar\'e duality group.
	It follows that the cohomology of $W_{OJ}=(\bD_\infty)^{|J|}$ with group ring coefficients is given by
	\[
		H^j(W_{OJ};\zz W_{OJ})=
		\begin{cases}
			\zz, &\text{if $j=|J|$,}\\
			0, &\text{otherwise}.
		\end{cases}
	\]
	Substituting this into the formula in Theorem~\ref{t:gpfree} gives the result.
\end{proof}

It is proved in \cite{dj00} that the cubical complex $\wt{T}_L$ can be identified with $\cu(W_{OL}, \ck)$, and that $W_{OL}$ and $A_L$ are commensurable.
Hence, Theorem~\ref{c:freeoct} gives a calculation of $H^*(A_L;\zz A_L)$.
(In fact this was the method used by Jensen and Meier in their proof of Theorem~\ref{t:jm}.)
\begin{Remark}
	Since $\bZ\subset \bd$, there is an obvious inclusion of graph products, $A_L\subset W_{OL}$.
	However, whenever $L$ is not a simplex, the image of $A_L$ is of infinite index in $W_{OL}$.
	In \cite{dj00} it is proved that $A_L$ and $W_{OL}$ are both isomorphic to subgroups of index $2^{|S|}$ in a larger RACG.
\end{Remark}

\paragraph{Weighted $L^2$-cohomology of $\boldsymbol{W_\OL}$.} We have $OS=\{s^+,s^-\}_{s\in S}$.
Let $\bq=(q_{s^\pm})_{s\in S}$.
The growth series of the infinite dihedral group is easy to compute.
(For example, see \cite[Ex.~17.1.2]{dbook}.) We have
\begin{equation*}
	V_s(\bq)=\frac{(1+q_{s^+})(1+q_{s^-})}{1-q_{s^+}q_{s^-}}\quad\text{and}\quad \frac{1}{V_s(\bq^\minus)} = \frac{q_{s^+}q_{s^-}-1}{(1+q_{s^-})(1+q_{s^+})},
\end{equation*}
and
\[
	p_s=V_s(\bq)-1 =\frac{q_{s^+}+q_{s^-}+2q_{s^+}q_{s^-}}{1-q_{s^+}q_{s^-}}.
\]

Write $\bq<\bone$ (resp.\ $\bq>\bone$) to mean that each $q_{s^\ga} <1$ (resp.\ $>1$), for $\ga\in \{+,-\}$.
The following is a corollary of the results in \S\ref{s:lgp}.
\begin{theorem}\label{c:oct}
	Suppose $(W_\OL,OS)$ is the octahedralization of $(W,S)$.
	\begin{enumeratei}
		\item\label{c:octgrowth}
		\[
			W_\OL(\bq)=W(\bp).
		\]
		\item\label{c:octsmall}
		If $\bq<\bone$, then
		\[
			\ltwo b^n(W_\OL)=\ltwop b^n(W_L).
		\]
		If, in addition, $\bp^\minus\in \ol{\car(W)}$ (i.e., $\bq$ is sufficiently close to $\bone$), then
		\[
			\ltwo b^j (W_\OL)=\sum_{J\in \cs(W,S)} b^{k}(K,K^{S-J})\dim_{\cnp} D^J,
		\]
		where a formula for $\dim_{\cnp} D^J$ is defined in Theorem~\ref{t:ddjodecomp}.
		\item\label{c:octthird}
		If $\bq={\mathbf 1/3}$, then
		\[
			L^2_{1/3}b^n(W_\OL)=L^2b^n(W_L),
		\]
		\item\label{c:octlarge}
		If $\bq>\bone$, then
		\[
			\ltwo b^n (W_\OL)=\sum_{J\in \cs(W,S)} b^{n-|J|}(K_J,\partial K_J) \prod_{s\in J} \frac{q_{s^+}q_{s^-}-1}{(1+q_{s^-})(1+q_{s^+})}.
		\]
		\item\label{c:octone}
		If $\bq=\bone$, then
		\[
			L^2b^n(W_\OL)=b^n(K,\partial K)=\ol{b}^{n-1}(L),
		\]
		where $\ol{b}^*(\ )$ refers to the reduced Betti-number.
	\end{enumeratei}
\end{theorem}
\begin{proof}
	(\ref{c:octgrowth}),(\ref{c:octsmall}) and (\ref{c:octthird}) are immediate from the formula for $p_s$, Lemma~\ref{l:growth}, and Theorems~\ref{c:main} and \ref{t:ddjo}.

	The region of convergence of the dihedral group $\car (\bD_\infty)$ is given by $q_{s^+}q_{s^-}<1$.
	For a spherical $J$, $V_J$ is the $J$-fold product of $\bD_\infty$.
	Thus, if $\bq>\bone$, then $\ltwo H^*(V_J)$ is concentrated in degree $|J|$ and
	\[
		\ltwo b^{|J|} (V_J)=\frac{1}{V_J(\bq^\minus)} = \prod_{s\in J}\frac{q_{s^+}q_{s^-}-1}{(1+q_{s^-})(1+q_{s^+})}.
	\]
	and we apply Theorem~\ref{t:largeq} to obtain (\ref{c:octlarge}).
	Finally, if $\bq=\bone$, then all the terms with nonempty $J$ in (\ref{c:octlarge}) (or (\ref{c:octsmall})) vanish and we obtain (\ref{c:octone}).
	(Since $L^2b^n(W_\OL)=L^2b^n(A_L)$, it also follows from Theorem~\ref{t:dl}.)
\end{proof}
\begin{remark}
	If $\bq$ is such that $q_{s^+}=q_{s^-}$ ($=q_s$), then in Corollary~\ref{c:oct}(i) the formula becomes
	\[
		\ltwo b^n (W_\OL)=\sum_{J\in \cs} b^{n-|J|}(K_J,\partial K_J)\prod_{s\in J} \, \frac{q_s-1}{1+q_s},
	\]
	and in (ii) the formula for $p_s$ simplifies to
	\[
		p_s=\frac{2q_s}{1-q_s}.
	\]
\end{remark}
\begin{remark}
	Our original motivation for computing the weighted $L^2$-co\-ho\-mo\-lo\-gy of the octahedralization of $W_L$ was to compute the ordinary $L^2$-cohomology of the Bestvina--Brady group $BB_L$ (which we did by the spectral sequence method in Theorem~\ref{t:lraag}).
	Although we were never able to make complete sense of the calculation, it was supposed to go something like this.
	Suppose that the multiparameter $\bq$ is a positive constant $q$.
	Then it should be possible to define the weighted $L^2$-cohomology of $A_L$ and $BB_L$.
	Moreover, the $\Ltwo$-Betti numbers of $A_L$ should equal those of $W_\OL$ and $\Ltwo$-Betti numbers of $BB_L$ should behave as if the Davis complex for $W_\OL$ were to split as a product of the complex $Z_L$ with the real line (with $\bD_\infty$-action) and as if the K\"unneth formula were true, i.e.,
	\begin{align*}
		\Ltwo b^n (W_\OL)&=\Ltwo b^n(BB_L)\cdot\Ltwo b^0(\bD_\infty)= \Ltwo b^n(BB_L)\left[\frac{1-q}{q+1}\right], \ \ \text{for $q<1$}\\
		\Ltwo b^{n+1} (W_\OL)&=\Ltwo b^n(BB_L)\cdot\Ltwo b^1(\bD_\infty)= \Ltwo b^n(BB_L)\left[\frac{q-1}{q+1}\right], \ \ \text{for $q>1$}.
	\end{align*}
	These can be rewritten as
	\begin{align}
		\Ltwo b^n(BB_L)&=\Ltwo b^n (W_\OL)\left[\frac{q+1}{1-q}\right], \ \ \text{for $q<1$, and}\label{e:1}
		\\
		&=\Ltwo b^{n+1} (W_\OL)\left[\frac{q+1}{q-1}\right], \ \ \text{for $q>1$}.\label{e:2}
	\end{align}
	Next we want to find the ordinary $L^2$-Betti numbers of $BB_L$ by taking the limit of either of these formulas as $q\to 1$.
	Since
	\[
		W_I(p)=(1+p)^{|I|}=\left(\frac{2q}{1-q} +1\right)^{|I|}= \left(\frac{1+q}{1-q}\right)^{|I|},
	\]
	the formula in Theorem~\ref{t:ddjodecomp} becomes
	\begin{equation}\label{e:61}
		\dim_{\cnp}D^J=(-1)^{|J|}\sum_{I\in \cs(W,S)_{\ge J}} \left(\frac{q-1}{1+q}\right)^{|I|}.
	\end{equation}
	After multiplying through by $(1-q)/(1+q)$, the only terms on the right hand side of \eqref{e:61} which will be nonzero when $q=1$ are those with $|I|=1$ and hence, $|J|=0$ or $1$.
	Since $\partial K= K^S$ is acyclic, the term for $J=\emptyset$ in Theorem~\ref{c:oct}, namely, $b^n(K,K^{S})$, vanishes.
	So, using Theorem~\ref{c:oct}~(i), formula \eqref{e:1} gives, for $q<1$,
	\[
		L^2b^n(BB_L)=\lim_{q\to 1} \Ltwo b^n(BB_L)=\sum_{s\in S}b^n(K,K^{S-s})= \sum_{s\in S}b^n(K_s,\partial K_s)
	\]
	where the last equation follows from the fact that $\partial K$ is acyclic and the excision, $H^*(\partial K,K^{S-s})\cong H^*(K_s,\partial K_s)$.
	Similarly, for $q>1$, Theorem~\ref{c:oct}~(ii) can be written as
	\[
		\Ltwo b^{n+1} (W_\OL)=\sum_{J\in \cs(W,S)} b^{n-|J|+1}(K_J,\partial K_J) \left(\frac{q-1}{1+q}\right)^{|J|}
	\]
	and \eqref{e:2} gives, for $q>1$,
	\[
		L^2b^n(BB_L)=\lim_{q\to 1} \Ltwo b^n(BB_L)=\sum_{s\in S}b^n(K_s,\partial K_s),
	\]
	since when $q=1$ only the terms with $|J|=1$ are nonzero.
	So, in Theorem~\ref{c:oct} both the formulas (i) and (ii) give the same answer as Theorem~\ref{t:lraag}.
\end{remark}

\section{Duality groups}\label{ss:duality}

An $m$-dimensional simplicial complex $L$ is \emph{Cohen--Macaulay} if for each $J\in \cs(L)$, $\ol{H}^*(\Lk(J))$ is concentrated in degree $m-|J|$ and is torsion-free.
For $J=\emptyset$, this means that $\ol{H}^*(L)$ is concentrated in degree $m$.
It also implies that any maximal simplex $J$ has dimension $m$, since, when $J$ is maximal, $\Lk(J)=\emptyset$ and our convention is that $\rH^*(\emptyset)$ is concentrated in degree $-1$ (where it is $=\zz$).

A group $G$ of type FP is an $n$-dimensional \emph{duality group} if $H^*(G;\zz G)$ is concentrated in degree $n$ and is torsion-free.

An immediate consequence of Corollary~\ref{c:freeoct} is the following.
\begin{proposition}[Brady--Meier \cite{bradym} and Jensen--Meier \cite{jm}]\label{p:duality}
	Suppose $(W,S)$ is a RACS with nerve $L$.
	Then the octahedralization, $W_\OL$, is a virtual duality group if and only if $L$ is Cohen--Macaulay (and consequently, the same is true for the associated RAAG, $A_L$).
\end{proposition}

Brady and Meier asked if these conditions are equivalent for a general Artin group (Question 2 of \cite{bradym}) and attributed the question to the first author of this paper.
The equivalence follows immediately from Theorem~\ref{t:salvetti} whenever the $K(\pi,1)$ Conjecture holds for $A$.
We state this as the following.
\begin{proposition}\label{p:salvetti}
	Suppose $X$ is the Salvetti complex associated to a Coxeter system with nerve $L$.
	Then $H^*(X;\zz A)$ is concentrated in degree $n$ and is torsion-free if and only if $L$ is an $(n-1)$-dimensional Cohen--Macaulay complex.
\end{proposition}
\noindent(The "if" direction was also proved by Brown--Meier in \cite{brownm} by using a different spectral sequence.)

Similarly, by Theorem~\ref{t:bb}, for Bestvina--Brady groups we have the following.
\begin{proposition}\label{bb:duality}
	Suppose $L$ is an acyclic flag complex.
	Then $BB_L$ is a duality group if and only if $L$ is Cohen--Macaulay.
	(For example, $L$ could be a acyclic, compact manifold with boundary.)
\end{proposition}

As explained in \cite[Sec.~6]{dm}, for graph products of finite groups, Theorem~\ref{t:ddjo2} leads to a slightly different condition.
An $m$-dimensional simplicial complex $L$ has \emph{punctured homology concentrated in dimension $m$} (abbreviated $PH^m$) if for each closed simplex $\gs$ of $L$, $\ol{H}_*(L-\gs)$ is torsion free and concentrated in degree $m$.
The $PH^m$ condition implies that $L$ is Cohen--Macaulay but is not equivalent to it (cf.~\cite[Cor.~6.9]{dm}).
\begin{proposition}[{\cite[Theorem~6.2]{dm}} and also cf.~\cite{hm}]\label{p:gpduality}
	Let $G$ be the graph product of a collection of nontrivial finite groups.
	Then $G$ is a $n$-dimensional duality group if and only if $L$ is $PH^{n-1}$.
\end{proposition}

\end{document}